\documentclass[reqno]{amsart}
\usepackage{comment,amssymb,latexsym,upref,enumerate,fouridx}
\usepackage{mathrsfs,color}
\usepackage{centernot}

\usepackage[colorlinks,linkcolor=blue,citecolor=blue]{hyperref}
\allowdisplaybreaks
\numberwithin{equation}{section}

\theoremstyle{plain}
\newtheorem{thm}{Theorem}[section]
\newtheorem{lem}[thm]{Lemma}
\newtheorem{prop}[thm]{Proposition}
\newtheorem{cor}[thm]{Corollary}

\theoremstyle{definition}
\newtheorem{Def}[thm]{Definition}

\theoremstyle{remark}
\newtheorem{rem}[thm]{Remark}

\newcommand{\ep}{\epsilon}
\newcommand{\gsl}{\fsl{g}}

\DeclareMathAlphabet{\mathpzc}{OT1}{pzc}{m}{it}

\newcommand{\fsl}[1]{{\centernot{#1}}}  % must use centernot package

\DeclareMathOperator{\tr}{tr}

\DeclareMathOperator{\supp}{supp}

\newcommand{\id}{\mathord{{\mathrm 1}\kern-0.27em{\mathrm I}}\kern0.35em}

\newcommand{\del}[1]{\partial_{#1}}
\newcommand{\delb}[1]{\bar{\partial}_{#1}}

\newcommand{\nablasl}[1]{\fsl{\nabla}{}_{#1}}

\DeclareMathOperator{\Ord}{O}

\newcommand{\oset}[2]{\overset{#1}{#2}{}}

\newcommand{\dsp}{\displaystyle}

\newcommand{\AND}{{\quad\text{and}\quad}}

\newcommand{\norm}[1]{\|#1\|}

\newcommand{\ipe}[2]{ ( #1 | #2 )}

\newcommand{\ip}[2]{ \langle #1 | #2 \rangle}

\newcommand{\Ac}{\mathcal{A}{}}

\newcommand{\Bc}{\mathcal{B}{}}

\newcommand{\Cc}{\mathcal{C}{}}

\newcommand{\Dc}{\mathcal{D}{}}

\newcommand{\Fc}{\mathcal{F}{}}
\newcommand{\fc}{\mathpzc{f}{}}
\newcommand{\Gc}{\mathcal{G}{}}
\newcommand{\gc}{\mathpzc{g}{}}
\newcommand{\Hc}{\mathcal{H}{}}
\newcommand{\hc}{\mathpzc{h}{}}
\newcommand{\Ic}{\mathcal{I}{}}

\newcommand{\Jc}{\mathcal{J}{}}

\newcommand{\Kc}{\mathcal{K}{}}

\newcommand{\pc}{\mathpzc{p}{}}
\newcommand{\Qc}{\mathcal{Q}{}}
\newcommand{\qc}{\mathpzc{q}{}}
\newcommand{\Rc}{\mathcal{R}{}}
\newcommand{\rc}{\mathpzc{r}{}}
\newcommand{\Sc}{\mathcal{S}{}}

\newcommand{\zc}{\mathpzc{z}{}}

\newcommand{\ab}{\bar{a}{}}

\newcommand{\bb}{\bar{b}{}}
\newcommand{\Cb}{\bar{C}{}}
\newcommand{\cb}{\bar{c}{}}

\newcommand{\db}{\bar{d}{}}
\newcommand{\eb}{\bar{e}{}}

\newcommand{\gb}{\bar{g}{}}

\newcommand{\Ib}{\bar{I}{}}

\newcommand{\Jb}{\bar{J}{}}
\newcommand{\Lb}{\bar{L}{}}

\newcommand{\Mb}{\bar{M}{}}

\newcommand{\rb}{\bar{r}{}}

\newcommand{\tb}{\bar{t}{}}

\newcommand{\ub}{\bar{u}{}}

\newcommand{\vb}{\bar{v}{}}

\newcommand{\wb}{\bar{w}{}}

\newcommand{\xb}{\bar{x}{}}

\newcommand{\Sigmab}{\bar{\Sigma}{}}

\newcommand{\nablab}{\bar{\nabla}{}}

\newcommand{\at}{\tilde{a}{}}

\newcommand{\bt}{\tilde{b}{}}

\newcommand{\ct}{\tilde{c}{}}

\newcommand{\ft}{\tilde{f}{}}

\newcommand{\gt}{\tilde{g}{}}

\newcommand{\Rt}{\tilde{R}{}}

\newcommand{\ut}{\tilde{u}{}}

\newcommand{\nablat}{\tilde{\nabla}{}}

\newcommand{\xit}{\tilde{\xi}{}}

\newcommand{\ah}{\hat{a}{}}

\newcommand{\xh}{\hat{x}{}}

\newcommand{\rf}{\mathfrak{r}{}}

\newcommand{\ev}{\mathbf{e}{}}

\newcommand{\Mbb}[1]{{\mathbb{M}_{#1\times #1}}{}}

\newcommand{\Pbb}{\mathbb{P}{}}

\newcommand{\Rbb}{\mathbb{R}{}}

\newcommand{\Tbb}{\mathbb{T}{}}

\newcommand{\Vbb}{\mathbb{V}{}}
\newcommand{\Wbb}{\mathbb{W}{}}

\newcommand{\Zbb}{\mathbb{Z}{}}

\newcommand{\Asc}{\mathscr{A}{}}

\newcommand{\Fsc}{\mathscr{F}{}}
\newcommand{\Gsc}{\mathscr{G}{}}

\newcommand{\Isc}{\mathscr{I}{}}

\newcommand{\Lsc}{\mathscr{L}{}}

\newcommand{\Ftt}{\mathtt{F}{}}
\newcommand{\ftt}{\mathtt{f}{}}

\newcommand{\Ntt}{\mathtt{N}{}}

\newcommand{\ptt}{\mathtt{p}{}}

\newcommand{\qtt}{\mathtt{q}{}}

\newcommand{\Vtt}{\mathtt{V}{}}

\begin{document}

\title[A Fuchsian viewpoint on the weak null condition]{A Fuchsian viewpoint on the weak null condition}

\author[T.A. Oliynyk]{Todd A. Oliynyk}
\address{School of Mathematical Sciences\\
9 Rainforest Walk\\
Monash University, VIC 3800\\ Australia}
\email{todd.oliynyk@monash.edu}

\author[J.A. Olvera-Santamar\'{i}a]{J. Arturo Olvera-Santamar\'{i}a}
\address{School of Mathematical Sciences\\
9 Rainforest Walk\\
Monash University, VIC 3800\\ Australia}
\email{arturo.olvera@monash.edu }

\begin{abstract}
We analyze systems of semilinear wave equations in $3+1$ dimensions whose associated asymptotic equation admit bounded solutions for suitably small
choices of initial data. Under this special case of the weak null condition, which we refer to as the \textit{bounded weak null condition}, we prove the existence of solutions to these systems of wave equations on neighborhoods of spatial infinity under a small initial data assumption. Existence is established using the Fuchsian
method. This method involves transforming the wave equations into a Fuchsian equation defined on a bounded spacetime region. The existence of solutions to the Fuchsian equation then follows from an application of
the existence theory developed in \cite{BOOS:2020}. This, in turn, yields, by construction, solutions to the original system of
wave equations on a neighborhood of spatial infinity.
\end{abstract}

\maketitle

\section{Introduction\label{intro}}
In this article, we establish global existence results for systems of semilinear wave equations in $3+1$ dimensions that
satisfy a weak null condition. Specifically, the class of semilinear wave equations that we consider are of the form\footnote{See Appendix \ref{indexing} for our indexing conventions.}
\begin{equation}  
\gb^{\alpha\beta}\nablab_\alpha \nablab_\beta \ub^K = \ab^{K\alpha\beta}_{IJ}\nablab_\alpha \ub^I \nablab_\beta \ub^J \label{Mbwave} 
\end{equation}
where the $u^I$, $1\leq I \leq N$, are a collection of scalar fields, the $\ab_{IJ}^{K}=\ab_{IJ}^{K\alpha\beta}\delb{\alpha}\otimes\delb{\beta}$, $1\leq I,J,K \leq N$, are prescribed smooth (2,0)-tensors fields on $\Rbb^4$, and $\nablab$ is the Levi-Civita connection of the Minkowski metric 
$\gb=\gb_{\alpha\beta}d\xb^\alpha\otimes d\xb^\beta$ on $\Rbb^4$.
We find it convenient to work throughout this article primarily in spherical coordinates
\begin{equation*}
(\xb^\mu) = (\xb^0,\xb^1,\xb^2,\xb^3)=(\tb,\rb,\theta,\phi)
\end{equation*}
in which the Minkowski metric is given by
\begin{equation} \label{gbdef}
\gb =  -d\tb\otimes d\tb + d\rb \otimes d\rb + \rb^2\gsl
\end{equation}
where
\begin{equation}\label{gsldef}
\gsl = d\theta \otimes d\theta + \sin^2(\theta) d\phi \otimes d\phi
\end{equation}
is the canonical metric on the $2$-sphere $\mathbb{S}^2$. 
For simplicity\footnote{This is certainly not necessary, and it is straightfoward to verify that all the results
of this article can be generalized to allow non-covariantly constant  tensors $\ab_{IJ}^{K}$ provided that they
satisfy suitable asymptotics.}, we assume for the remainder of the article that the tensor fields $\ab^K_{IJ}$ are covariantly constant, i.e.  $\nablab \ab_{IJ}^{K}=0$, which is equivalent to the condition that the components of $\ab_{IJ}^{K}$ in a Cartesian coordinate
system $(\xh^\mu)$ are constants, that is, $\ab_{IJ}^{K} = \ah_{IJ}^{K\alpha\beta}\hat{\partial}_{\alpha}\otimes\hat{\partial}_{\beta}$
for some set of constant coefficients $ \ah_{IJ}^{K\alpha\beta}$.

In order to define the weak null condition that we will consider in this article, we first introduce the out-going null one-form
$\Lb=-d\tb + d\rb$
and set
\begin{equation} \label{bbIJKdef}
\bb^K_{IJ} := \ab^{K\mu\nu}_{IJ}\Lb_\mu \Lb_\nu = \ab_{IJ}^{K00}-\ab_{IJ}^{K01}-\ab_{IJ}^{K10}+\ab_{IJ}^{K11}.
\end{equation}
As we show below, see \eqref{bbIJKformula}, the
$\bb^K_{IJ}$  define smooth functions on $\mathbb{S}^2$. We use these functions to define 
the \textit{asymptotic equation} associated
to \eqref{Mbwave} by
\begin{equation}\label{asympeqn}
(2-t)\del{t}\xi =  \frac{1}{t}Q(\xi)
\end{equation}
where $\xi=(\xi^K)$ and
\begin{equation}\label{Qmap}
Q(\xi) = (Q^K(\xi)):=(-2 \chi(\rho) \rho^m \bb^K_{IJ} \xi^I \xi^J). 
\end{equation}
In this equation, $t$ and $\rho$ are coordinates that arise from a compactification of a neighborhood of spatial infinity, see Section \ref{cylinder:sec}
and equation \eqref{rhodef} for details, while $\chi(\rho)$ is a smooth cut-off function. Furthermore, the time coordinate $t$ is chosen so that $0<t\leq 1$ and
and $t=0$ corresponds to future null-infinity. We remark that this type of equation was first introduced by H{\"o}rmander \cite{Hormander:1987,Hormander:1997}  to analyze the blow-up time for wave equations that do not satisfy the null condition of  Klainerman \cite{Klainerman:1980}, which in our notation is defined by the vanishing of the $\bb^K_{IJ}$.

The weak null condition, which was first introduced in \cite{LindbladRodnianski:2003}, is a growth condition on solutions of the asymptotic
equation, namely that solutions of \eqref{asympeqn} satisfy a bound of the form $|\xi(t)|\lesssim t^{-C\ep}$ for some fixed constant $C>0$ and  initial data at $t=1$ satisfying $|\xi(1)|\leq \ep\leq \ep_0$ for $\ep_0>0$ sufficiently small. It is still an open conjecture, even
in the semilinear setting, to determine if the weak null condition is enough to ensure the global existence of solutions under a suitable
small initial data assumption. In this article, we will consider the
following restricted weak null condition, which includes the classical null condition as a special case:

\begin{Def} \label{bwnc}
\noindent The asymptotic equation is said to satisfy the \textit{bounded weak null condition} if there exist constants $\Rc_0>0$ and $C>0$ such that solutions of the asymptotic initial value problem (IVP)
\begin{align}
(2-t)\del{t}\xi &= \frac{1}{t}Q(\xi), \label{asympprop1.1} \\
\xi|_{t=1} &= \mathring{\xi}, \label{asympprop1.2}
\end{align}
exist for $t\in (0,1]$ and are bounded by $\dsp \sup_{0<t\leq 1}|\xi(t)| \leq C$
for all initial data $\mathring{\xi}$ satisfying $|\mathring{\xi}| < \Rc_0$. 
\end{Def}

We remark here that Keir \cite{Keir:2019} has analyzed systems of quasilinear wave equations
with quadratic semilinear terms under a slightly stronger assumption that requires, in addition to the boundedness assumption, a stability condition on solutions to the asymptotic equation.
Under these conditions, Kerr was able to establish, using a generalization of the p-weighted energy method of Dafermos and
Rodnianski \cite{DafermosRodnianski:2010} that was developed in \cite{Keir:2018},  the
global existence of solutions to the future of a truncated outgoing characteristic hypersurface under a suitable small initial data
assumption. In particular, his results imply that semilinear systems of wave equations of the form \eqref{Mbwave} whose asymptotic
equations satisfy his boundeness and stability condition admit solution on spacetime regions of the form
$\{\,(\tb,\rb)\, |\, \tb>\max\{0,\rb-\rb_0\}, \rb\geq 0\, \} \times \mathbb{S}^2$
for suitably small initial data that is prescribed on the truncated null-cone
$\{\,(\tb,\rb)\,|\, \tb=\max\{0,\rb-\rb_0\}, \rb\geq 0\, \}\times\mathbb{S}^2$.

In light of Kerr's results, we will restrict our attention to establishing the existence of solutions to \eqref{Mbwave}
on neighborhoods of spatial infinity of the form  
\begin{equation} \label{Mbr0def}
\Mb_{r_0} = \bigl\{ \,(\tb,\rb)\, \bigl| \, 0<\tb < \rb - 1/ r_0, \; 1/r_0< \rb < \infty\, \bigr\} \times \mathbb{S}^2
\end{equation}
where $r_0>0$ is a positive constant and initial data is prescribed on the hypersurface
\begin{equation}  \label{Sigmabr0def}
\Sigmab_{r_0} = \bigl\{ \,(\tb,\rb)\, \bigl| \, \tb=0,\; 1/r_0 < \rb < \infty\, \bigr\} \times \mathbb{S}^2.
\end{equation}
This will compliment Kerr's results, at least in the semilinear setting, by establishing the existence of solutions on regions
not covered by his existence results. More important, in our opinion, is that we establish these global existence results
using a new method, called the \textit{Fuchsian method}, that we
believe will be prove useful, more generally, for the analysis of nonlinear wave equations. 
Informally, the main existence result of this article, see Corollary \ref{existcor} for the precise version, can be stated as follows:

\begin{thm} \label{informalthm}
Suppose $\zc>0$ and the asymptotic equation \eqref{asympeqn} associated to \eqref{Mbwave} satisfies the bounded weak null condition.
Then there exists a $r_0>0$ such that for suitably small initial data $\vb^K$, $\wb^K$ defined on $\Sigmab_{r_0}$, which does not have to be compactly supported, there exists a unique classical solution
$\ub^K\in C^2(\Mb_{r_0})$ to the initial value problem
\begin{align*}
\gb^{\alpha\beta}\nablab_\alpha \nablab_\beta \ub^K &= \ab^{K\alpha\beta}_{IJ}\nablab_\alpha \ub^I \nablab_\beta \ub^J
\quad \text{in $\Mb_{r_0}$,}\\
(\ub^K, \del{\tb}\ub^K) &= (\vb^K,\wb^K) \hspace{1.5cm} \text{in $\Sigmab_{r_0}$,}
\end{align*}
that satisfies the pointwise bound
\begin{equation*}
|\ub^K|\lesssim \frac{\rb}{\rb^2-\tb^2}\biggl(1-\frac{\tb}{\rb}\biggr)^{1-\zc} \quad \text{in $\Mb_{r_0}$.}
\end{equation*}
\end{thm}

\noindent We note that further $L^2$ and pointwise bounds for $\ub^K$ and its derivatives are easily determined from Corollary \ref{existcor}.

\subsection{Semilinear wave equations satisfying the bounded weak null condition}
In \cite{Keir:2019}, Kerr showed that systems of wave equations of the form \eqref{Mbwave} with 
\begin{equation} \label{abmodel}
\ab_{IJ}^{K\alpha\beta}= \Ib^{KL}\Cb_{LIJ}\delta^\alpha_0\delta^\beta_0,
\end{equation}
where $\Ib^{KL}$ is a constant, positive definite, symmetric matrix and the $\Cb_{LIJ}$ are any constants satisfying
\begin{equation}\label{Cbantisym}
\Cb_{LIJ}=-\Cb_{ILJ},
\end{equation} 
have associated asymptotic equations that satisfy the bounded weak null condition. This is, in fact, easy to verify
since the choice \eqref{abmodel} leads, by \eqref{bbIJKdef}-\eqref{Qmap}, to the associated asymptotic
equation
\begin{equation} \label{asymeqnmodel}
(2-t)\del{t}\xi^K = -\frac{2}{t}\chi(\rho)\rho^m \Ib^{KL}\Cb_{LIJ}\xi^I \xi^J.
\end{equation} 
Introducing the inner-product $\ipe{\xi}{\eta} = \check{\Ib}{}_{IJ}\xi^I \eta^J$,
where $(\check{\Ib}{}_{IJ})=(\Ib^{IJ})^{-1}$, and contracting \eqref{asymeqnmodel} with $\check{\Ib}{}_{LK}\xi^L$, we get
\begin{equation} \label{asympenergy}
(2-t)\ipe{\xi}{\del{t}\xi} = -\frac{2}{t}\chi(\rho)\rho^m \check{\Ib}{}_{LK}\Ib^{KM}\Cb_{MIJ}\xi^L\xi^I \xi^J=
 -\frac{2}{t}\chi(\rho)\rho^m \Cb_{LIJ}\xi^L\xi^I \xi^J \oset{\eqref{Cbantisym}}= 0.
\end{equation}
But this implies $\del{t}(\ipe{\xi}{\xi})=0$, and so, we conclude that any solution of the asymptotic IVP \eqref{asympprop1.1} -\eqref{asympprop1.1} exists
for all $t\in (0,1]$ and satisfies $\ipe{\xi(t)}{\xi(t)}=\ipe{\mathring{\xi}}{\mathring{\xi}}$.
Letting $|\cdot|$ denote the Euclidean norm, we then have that 
$\frac{1}{\sqrt{C}}|\cdot| \leq \sqrt{\ipe{\cdot}{\cdot}} \leq\sqrt{C}  |\cdot|$ for some constant $C>0$, and consequently, 
by the above inequality, we arrive at the bound
$\sup_{0<t\leq 1}|\xi(t)|  \leq C  |\mathring{\xi}|$, which verifies that the bounded weak null condition
is fulfilled.

The calculation \eqref{asympenergy} also shows that this class of semilinear equations satisfies the structural
condition from \cite{Katayama_et_al:2015} called \textit{Condition H}. Because of this,
the global existence results established in from \cite{Katayama_et_al:2015} apply and yield the existence of global solutions
to \eqref{Mbwave} on the region $\tb>0$ for suitably small initial data with compact support. We further note that due
to the compact support of the initial data, the results of \cite{Katayama_et_al:2015} can, in fact, be deduced as
a special case of the global existence theory developed in \cite{Keir:2019}, but do not apply to the situation we are considering
in this article because we allow for non-compact initial data in addition to a less restrictive weak null condition.

\subsection{Prior and related works}
Early global existence results for nonlinear wave equations in $3+1$ dimensions that violate the null condition but satisfy the weak null condition were established for quasilinear wave equations in \cite{Alinhac:2003,Lindblad:2008}, systems of semilinear wave equations
in \cite{Alinhac:2006}, and the Einstein equations in wave coordinates in \cite{LindbladRonianski:2005,LindbladRodnianski:2010}. 
More recent results can be found in \cite{Katayama_et_al:2015}, which we discussed above, for semilinear equations, and in the articles  \cite{Bingbing_et_al:2015,DengPusateri:2018,HidanoYokoyama:2018} and, as we discussed
above, 
in \cite{Keir:2018,Keir:2019} for quasilinear equations.

\subsection{The Fuchsian method}
Singular systems of hyperbolic equations that can be expressed in the form
\begin{equation}
B^0(t,u)\del{t}u + B^i(t,u)\nabla_{i} u  = \frac{1}{t}\Bc(t,u) u + F(t,u) \label{symivp}
\end{equation}
are said to be \textit{Fuchsian}.  Traditionally, these systems have been viewed as \emph{singular initial value problems} (SIVP) where \textit{asymptotic data} is prescribed at
the singular time $t=0$ and then \eqref{symivp} is used to evolve the asymptotic data \emph{away from the singular time} to construct solutions on time intervals $t\in (0,T_0]$ for some, possibly small, $T_0>0$. The SIVP for
Fuchsian equations has been studied by many mathematicians and has found a wide array of applications, for example, see
\cite{andersson2001,choquet-bruhat2006,choquet-bruhat2004,damour2002,heinzle2012,isenberg1999,isenberg2002,kichenassamy1998,OliynykKunzle:2002a,OliynykKunzle:2002b,Rendall:2004} for applications in the analytic setting and  \cite{ames2013a,ames2013b,ames2017,beyer2010b,beyer2011,beyer2017,claudel1998a,kichenassamy2007k,rendall2000,rendall2004a,stahl2002}
in the class of Sobolev regularity. 

While the SIVP approach for establishing the existence of solutions to \eqref{symivp} is useful for certain applications, there are many situations where
the \textit{global initial value problem} (GIVP) for the Fuchsian system \eqref{symivp} is the relevant problem to solve. In this case, initial data is prescribed at some finite time $t=T_0>0$, and the problem becomes
to establish the existence of solutions to \eqref{symivp} all the way up to the singular time at $t=0$, that is, for $t\in (0,T_0]$. 
The flavour of this problem is that of a global existence problem and the study of such problems was initiated by the first author in  \cite{Oliynyk:CMP_2016}. In that article, existence results were established for the Fuchsian GIVP, which were 
then used to deduce the
future nonlinear stability of perturbations of 
Friedmann-Lema\^{i}tre-Robertson-Walker (FLRW)
solutions to the Einstein-Euler equations with a positive cosmological constant and a linear equation of state $p=K\rho$ for $0<K\leq 1/3$. 
This result represents the first instance  of a new method for establishing the global existence
of solutions to systems of hyperbolic equations that we now refer to as the \emph{Fuchsian method}. Specifically, the goal of the Fuchsian method
is to transform a given system of hyperbolic equations into a GIVP for a suitable Fuchsian system, and then to
deduce the existence of solutions to the Fuchsian GIVP from general existence theorems such as those established in \cite{Oliynyk:CMP_2016}.

The advantage of the Fuchsian method is that solving the Fuchsian GIVP is technically much simpler compared to establishing global
existence results for the original system of hyperbolic equations. The most difficult part of applying the 
Fuchsian method is finding a suitable set of
variables and coordinates needed to bring the original system of hyperbolic equations into the required Fuchsian form, and this 
problem is typically more geometric-algebraic than analytic.
In recent years, the GIVP existence theory for Fuchsian systems has been further developed
in the articles \cite{LiuOliynyk:2018a,LiuOliynyk:2018b,BOOS:2020,FOW:2020}, and these existence results and those from
\cite{Oliynyk:CMP_2016} have been used to deduce the global existence of solutions for a number
of different systems of hyperbolic equations in the articles 
\cite{BOOS:2020,FOW:2020,LeFlochWei:2020,LiuOliynyk:2018b,LiuOliynyk:2018a,LiuWei:2019,Oliynyk:2020,Wei:2018}.

\subsection{Outlook and future work}
The results of this article can be generalized and extended in a number of ways, which we describe briefly here. First, 
the hierarchical weak null
condition defined in \cite{Keir:2018}, which is general enough to encompass all known results with the exception
of those from \cite{Katayama_et_al:2015,Keir:2019} and this article,  can be easily handled using the Fuchsian method. Indeed, by
introducing a re-scaling of the form $\Vtt^K = t^{\mu_K}V^K$ for a suitable choices of
the constants $\mu_K$, it is straightforward, assuming the hierarchical weak null condition, to bring the system \eqref{MwaveH}
into a Fuchsian form in terms of the variables $\Vtt^K$ that satisfies all the hypotheses of Theorem 3.8 of \cite{BOOS:2020}. An application of this theorem
then yields the existence
of solutions to the semilinear system of wave equations \eqref{Mbwave} on $\Mb_{r_0}$ for suitable initial data specified on $\Sigmab_{r_0}$. We further note that all of the results of this article also hold in a neighborhood of timelike infinity. This follows from replacing the cylinder
at spatial infinity construction from Section \ref{cylinder:sec} by an analogous cylinder at temporal infinity construction. 
With this change, the arguments in this article go through in a similar fashion. In particular, this type of result
in conjunction with the results of the current article can be combined to establish a global existence result for
\eqref{Mbwave} on the regions $\tb>0$ for initial data specified on the constant time hyperspace $\tb=0$ that, importantly, does
not have to be compactly supported. We will report on these extensions
in a separate article.

We further note that systems of quasilinear wave equations can also be handled via the Fuchsian method. However, this
requires an extension of the
Fuchsian GIVP existence theory from \cite{BOOS:2020} to allow for spatial manifolds with boundaries and the use of boundary weighted Sobolev spaces. 
This work will be reported on in a separate article that is currently in preparation. %\cite{Oliynyk:2021}. 
Finally, we note these types of weighted results are also of interest in the semilinear setting
because they allow for for more general choices of initial data.

\subsection{Overview}
To prove the main result of this article, we use the Fuchsian method. This involves a sequence of steps where 
we transform the wave equation 
\eqref{Mbwave} on the non-compact domain $\Mb_{r_0}$ into a Fuchsian equation on a compact
domain that is in a form that allows us to directly apply the Fuchsian GIVP existence theory from \cite{BOOS:2020}, and 
thereby, establish the existence of solutions on $\Mb_{r_0}$ to semilinear wave equations satisfying the bounded weak null condition.

The derivation of the Fuchsian equation begins in Section \ref{cylinder:sec} where we map the domain $\Mb$, see
\eqref{Mbdef}, onto 
the \textit{cylinder at spatial infinity} 
$(0,2)\times (0,\infty) \times \mathbb{S}^2$, which compactifies the outgoing null-rays. 
We 
then use this mapping in Section \ref{conformal:sec} to push-forward the wave equation \eqref{Mbwave} on $\Mb_{r_0}$.
This yields the conformal wave equation \eqref{MwaveA} (see also \eqref{MwaveB}) defined on the manifold with boundary $M_{r_0}$ (see \eqref{Mr0def}) whose closure is now compact. We then proceed, in
Section \ref{firstorder:sec}, to write
the conformal wave equation in the first order form \eqref{MwaveG} using the variables $V^K=(V^K_\Ic)$ defined by \eqref{V0Kdef}
and \eqref{V1def}. This choice of variables is motivated by the structure of the most singular terms appearing in the conformal wave equation (see \eqref{fKexp}) and the requirement that the resulting first order system be symmetric hyperbolic.
The first order equation \eqref{MwaveG} is then replaced, in Section \ref{extendedsec}, by the extended system
\eqref{MwaveI}. The point of the extended system is twofold. First, it is defined on an extended spacetime region of the form $(0,1)\times \Sc$ where $\Sc$ is now a closed manifold, which, unlike $M_{r_0}$ is needed to apply the existence theory developed
in \cite{BOOS:2020}, and second, its solutions yield solutions to \eqref{MwaveG} on $M_{r_0} \subset (0,1)\times \Sc$ that
are independent of the particular form of the initial data on the ``unphysical'' part of the initial hypersurface $\{1\}\times \Sc$
given by $\{1\}\times \Sc \setminus \Sigma_{r_0}$. Here, $\Sigma_{r_0}$ is the spatial hypersurface where the initial data is prescribed. 
The upshot of this is that we lose nothing by working with the extended
system \eqref{MwaveI} rather than the first order form of the conformal wave equation given by \eqref{MwaveG}. Next, we
differentiate the extended system \eqref{MwaveI} to obtain the system \eqref{MwaveK}, which can be viewed as an
evolution equation for the variables $W^K_{j}=t^{\kappa}(\Dc_j V^K)$, and we use the flow of
the asymptotic equation \eqref{asympeqn} to change variables from $V_0$ to a new variable $Y$ defined by \eqref{V0def}-\eqref{Ydef}. The point of this change of variables is that it
removes the most singular term from the ``time'' component of the extended system \eqref{MwaveI}, which results
in the evolution equation \eqref{MwaveM}. In Section \ref{complete:sec}, we then apply the projector $\Pbb$ to the extended system \eqref{MwaveI}
to obtain an equation for $X^K=t^{-\nu}\Pbb V^K$ given by \eqref{MwaveN}, and we combine the three systems 
\eqref{MwaveI}, \eqref{MwaveK} and \eqref{MwaveN} into the single Fuchsian system \eqref{MwaveO} to obtain
an evolution system for $Z=(W^K_j, X^K,Y^K)$. It is then shown in Section \ref{coefprops} that, under the flow
assumptions from Section \ref{Asymptoticflowassump}, the Fuchsian system \eqref{MwaveO} satisfies, for a suitable choice of the 
parameters $\kappa,\nu$, all the assumptions needed to apply the Fuchsian GIVP existence theory from \cite{BOOS:2020}. Applying Theorem 3.8. from \cite{BOOS:2020} then yields the GIVP result for the Fuchsian system
\eqref{MwaveO} that is stated in Theorem \ref{existthm}. This, in turn, yields, by construction, a small initial data
global existence result for the original system of wave equations \eqref{Mbwave}. Finally, we easily derive Corollary \ref{existcor} from Theorem \ref{existthm} using Proposition \ref{asympprop}, which is the main result of the article and is stated informally as
Theorem \ref{informalthm} in the Introduction.

\section{Conformal wave equations near spatial infinity} 
\subsection{The cylinder at spatial infinity\label{cylinder:sec}}
The first step in transforming the system of semilinear wave equations \eqref{Mbwave} into 
Fuchsian form is
to compactify the neighborhoods of spatial infinity defined by \eqref{Mbr0def}. To this  
end, we let
\begin{equation} \label{Mbdef}
\Mb = \{ \, (\tb,\rb) \in (-\infty,\infty)\times (0,\infty) \: |\: -\tb^2+\rb^2 > 0 \, \}\times \mathbb{S}^2
\end{equation} 
and follow \cite{FrauendienerHennig:2014}, see also \cite[\S 4.1.1.]{BOOS:2020}, by mapping $\Mb$ to
\begin{equation*}
M = (0,2)\times (0,\infty) \times \mathbb{S}^2
\end{equation*}
using the diffeomorphism
\begin{equation} \label{psibdef}
\psi \: :\: \Mb \longrightarrow M \: : \: (\xb^\mu) = (\tb,\rb,\theta,\phi) \longmapsto (x^\mu) = \biggl(1-\frac{\tb}{\rb},\frac{\rb}{-\tb^2+\rb^2},\theta,\phi\biggr),
\end{equation}
where we label the coordinates on $M$ by
\begin{equation*}
(x^\mu) = (x^0,x^1,x^2,x^3) = (t,r,\theta,\phi).
\end{equation*}
The inverse of the mapping \eqref{psibdef} is readily obtained by solving
\begin{equation} \label{tbrb2tr}
t =1-\frac{\tb}{\rb}, \quad r =\frac{\rb}{-\tb^2+\rb^2},
\end{equation}
for $(\tb,\rb)$ to get
\begin{equation} \label{tr2tbrb}
\tb =\frac{1-t}{rt(2-t)}, \quad \rb =\frac{1}{tr(2-t)}.
\end{equation}

The importance of the diffeomorphism \eqref{psibdef} is that it defines a compactification of null-rays in $\Mb$. Decomposing the boundary of $M$ as
\begin{equation*}
\del{}M = \Isc^+ \cup i^0 \cup \Isc^{-},
\end{equation*}
where
\begin{equation*}
\Isc^{+} = \{0\}\times (0,\infty) \times \mathbb{S}^2, \quad  \Isc^{-} = \{2\}\times (0,\infty) \times \mathbb{S}^2 \AND i^0 = [0,2]\times \{0\} \times \mathbb{S}^2,
\end{equation*}
the boundary components $\Isc^{\pm}$ correspond to portions of ($+$) future and ($-$) past null-infinity, respectively, while 
$i^0$ corresponds to spatial infinity.
Furthermore, the  spacelike hypersurface 
$\{1\} \times (0,\infty) \times \mathbb{S}^2$ in $M$
corresponds to the constant time hypersurface $\tb=0$ 
in Minkowski spacetime. 

\begin{rem} The method used above to obtain the Lorentzian manifold $(M,g)$ from $(\Mb,\gb)$
is an example of the \textit{cylinder at spatial infinity} construction that was first introduced by Friedrich in \cite{Friedrich:JGP_1998} as a tool to
analyze the behavior of his conformal  version of the Einstein field equations, see \cite{Friedrich1981,Friedrich:1986,Friedrich:1991}, near spatial infinity. For further applications of this construction to linear wave and spin-2 equations on Minkowski spacetime, see the articles \cite{Beyer_et_al:2014,DoulisFrauendiener:2013,FrauendienerHennig:2014,Friedrich:2002,MacedoKroon:2018}.
\end{rem}

\subsection{Expansion formulas for the tensor components $\ab_{IJ}^{K\alpha\beta}$}
Before proceeding with the transformation to Fuchsian form, we  first derive an expansion formula for the tensor components $\ab_{IJ}^{K\alpha\beta}$ that will
play an important role in the calculations that follow. To start, we compute the Jacobian matrix
\begin{equation} \label{Jbdef}
(\Jb^\alpha_\mu) =\begin{pmatrix}
 1 & 0 & 0 & 0 \\
 0 & \sin (\theta ) \cos (\phi ) & \sin (\theta ) \sin (\phi ) & \cos (\theta ) \\
 0 & \frac{\cos (\theta ) \cos (\phi )}{\rb} & \frac{\cos (\theta ) \sin (\phi )}{\rb} & -\frac{\sin (\theta
   )}{\rb} \\
 0 & -\frac{\csc (\theta ) \sin (\phi )}{\rb} & \frac{\csc (\theta ) \cos (\phi )}{\rb} & 0
\end{pmatrix}
 \end{equation}
from the  change of variables
$(\xh^\mu)=(\tb,\rb\cos(\phi)\sin(\theta),\rb\sin(\phi)\sin(\theta),\rb\cos(\theta))$
from spherical to Cartesian coordinates. Using this and the tensorial transformation law
\begin{equation} \label{abcomponents}
\ab_{IJ}^{K\alpha\beta} = \Jb^\alpha_\mu \ah_{IJ}^{K\mu\nu} \Jb^\beta_\nu,
\end{equation}
we can expand the components \eqref{abcomponents} in powers of $\rb$ as
\begin{equation} \label{abexp}
\ab_{IJ}^{K\alpha\beta} = \frac{1}{\rb^2}\eb_{IJ}^{K\alpha\beta}+
\frac{1}{\rb}\db_{IJ}^{K\alpha\beta}+ \cb_{IJ}^{K\alpha\beta}
\end{equation}
where the expansions coefficients can be used to define the following geometric objects (see Appendix \ref{indexing} for
our indexing conventions) on $\mathbb{S}^2$: 
\begin{enumerate}
\item[(a)] smooth functions
$\eb_{IJ}^{K \pc \qc }$,  $\db_{IJ}^{K \pc \qc}$ and $\cb_{IJ}^{K \pc \qc}$,
\item[(b)] smooth vector fields 
$\eb_{IJ}^{K \qc \Lambda}$, $\eb_{IJ}^{K \Lambda\qc}$,  $\db_{IJ}^{K \qc \Lambda}$,
$\db_{IJ}^{K \Lambda\qc}$,  $\cb_{IJ}^{K \qc \Lambda}$, and $\cb_{IJ}^{K \Lambda\qc}$,
\item[(c)] and smooth (2,0)-tensor fields  $\eb_{IJ}^{K \Lambda \Sigma }$,  $\db_{IJ}^{K  \Lambda \Sigma}$ and $\cb_{IJ}^{K  \Lambda \Sigma}$.
\end{enumerate}

The only terms of the expansion \eqref{abexp} that we will need to consider in any detail are the $\cb_{IJ}^{K\alpha\beta}$. Now, it can be easily verified that the
\textit{non-vanishing}  $\cb_{IJ}^{K\alpha\beta}$ are given by 
\begin{align}
\cb_{IJ}^{K00} &= \ah_{IJ}^{K00} , \label{cb00} \\
\cb_{IJ}^{K01} &= \sin (\theta ) (\ah_{IJ}^{K01}  \cos (\phi )+\ah_{IJ}^{K02}  \sin
   (\phi ))+\ah_{IJ}^{K03}  \cos (\theta ), \label{cb01} \\
\cb_{IJ}^{K10} &= \sin (\theta ) (\ah_{IJ}^{K10} \cos (\phi )+\ah_{IJ}^{K20}  \sin (\phi
   ))+\ah_{IJ}^{K30}  \cos (\theta ) \label{cb10}
\intertext{and}
\cb_{IJ}^{K11} &=\sin ^2(\theta ) \left(\ah_{IJ}^{K11}  \cos ^2(\phi
   )+(\ah_{IJ}^{K12} +\ah_{IJ}^{K21}) \sin (\phi ) \cos (\phi )+\ah_{IJ}^{K22} \sin
   ^2(\phi )\right) \notag \\
   &\qquad +\sin (\theta ) \cos (\theta ) ((\ah_{IJ}^{K13}+\ah_{IJ}^{K31}) \cos (\phi
   )+(\ah_{IJ}^{K23}+\ah_{IJ}^{K32}) \sin (\phi ))+\ah_{IJ}^{K33} \cos ^2(\theta ). \label{cb11} 
\end{align}
Furthermore, with the help of \eqref{Jbdef} and \eqref{abcomponents}, we find via a straightforward calculation that the $\bb^K_{IJ}$, which are defined by \eqref{bbIJKdef}, can be expressed as
 \begin{align} \label{bbIJKformula}
 \bb^K_{IJ}&=\ah_{IJ}^{K00}-\sin (\theta ) (\ah_{IJ}^{K01} \cos (\phi )+\ah_{IJ}^{K02} \sin (\phi ))-\ah_{IJ}^{K03} \cos (\theta )-\sin (\theta )
   (\ah_{IJ}^{K10} \cos (\phi )+\ah_{IJ}^{K20} \sin (\phi ))\notag \\
   &\qquad +\sin ^2(\theta ) \left(\ah_{IJ}^{K11} \cos ^2(\phi
   )+(\ah_{IJ}^{K12}+\ah_{IJ}^{K21}) \sin (\phi ) \cos (\phi )+\ah_{IJ}^{K22} \sin ^2(\phi )\right) \notag \\
   &\qquad +\sin (\theta ) \cos (\theta )
   ((\ah_{IJ}^{K13}+\ah_{IJ}^{K31}) \cos (\phi )+(\ah_{IJ}^{K23}+\ah_{IJ}^{K32}) \sin (\phi ))-\ah_{IJ}^{K30} \cos (\theta
   )+\ah_{IJ}^{K33}\cos ^2(\theta ) .
\end{align}

\subsection{The conformal wave equation\label{conformal:sec}}
Letting
\begin{equation*}% \label{gt2gb}
\gt =\psi_* \gb
\end{equation*}
denote the push-forward of the Minkowski metric \eqref{gbdef} from $\Mb$ to $M$ using the map \eqref{psibdef}, we find after a routine calculation that
\begin{equation} \label{Mctrans}
\gt = \Omega^2 g
\end{equation}
where
\begin{equation}\label{MOmegadef}
\Omega = \frac{1}{r(2-t)t}
\end{equation}
and
\begin{equation} \label{Mgdef}
g= -dt\otimes dt + \frac{1-t}{r}(dt\otimes dr + dr\otimes dt) +\frac{(2-t)t}{r^2}dr\otimes dr +\gsl.
\end{equation}
Using the map \eqref{psibdef} to push-forward the wave equations \eqref{Mbwave} yields the system
of wave equations
\begin{equation}  
\gt^{\alpha\beta}\nablat_\alpha \nablat_\beta \ut^K = \at^{K\alpha\beta}_{IJ}\nablab_\alpha \ut^I \nablat_\beta \ut^J \label{Mtwave} 
\end{equation}
where $\nablat_\alpha$ is the Levi-Civita connection of the metric $\gt_{\alpha\beta}$,
\begin{align} \label{utdef}
\ut^K &= \psi_* \ub^K
\intertext{and}
\at^{K\alpha\beta}_{IJ} &= \psi_*(\ab^{K}_{IJ})^{\alpha\beta}. \label{atdef}
\end{align}
Since $M=\psi(\Mb)$, it is clear the original system of semilinear wave equations \eqref{Mbwave} on $\Mb$ are completely equivalent to \eqref{Mtwave} on $M$.

Next, we observe that the Ricci scalar curvature of $\gt_{\alpha\beta}$ vanishes by virtue of $\gt_{\alpha\beta}$ being the push-forward of the Minkowski metric. Furthermore, a straightforward calculation using \eqref{Mgdef} shows that
the Ricci scalar of the metric $g_{\alpha\beta}$ also vanishes. Consequently, it follows from the formulas \eqref{utrans}-\eqref{wavetransC}  and \eqref{ftransB}-\eqref{ftransC}, with $n=4$, from Appendix \ref{ctrans}
that the system of wave equations \eqref{Mtwave} transform under the conformal transformation \eqref{Mctrans} into
\begin{equation} \label{MwaveA}
g^{\alpha\beta}\nabla_\alpha \nabla_\beta u^K = f^K
\end{equation}
where $\nabla$ is the Levi-Civita connection of $g$,
\begin{equation}\label{utK2uK}
\ut^K = rt(2-t)u^K
\end{equation}
and
\begin{align}
f^K = \at_{IJ}^{K\mu\nu}\biggl(\frac{1}{rt(2-t)}\nabla_\mu u^I\nabla_\nu u^J &
+\frac{1}{(rt(2-t))^2}\bigl(\nabla_\mu(rt(2-t)) u^I \nabla_\nu u^J+\nabla_\mu u^I \nabla_\nu(rt(2-t)) u^J \bigr) \notag
\\
&+ \frac{1}{(rt(2-t))^3}\nabla_\mu\bigl(rt(2-t)\bigr) \nabla_\nu(rt(2-t)) u^I u^J\biggr). \label{fKdef}
\end{align}
We will refer to this system as the \textit{conformal wave equations}.

A routine computation involving \eqref{Mgdef} then shows that conformal wave equations \eqref{MwaveA} can be expressed as
\begin{equation} \label{MwaveB}
(-2+t)t\del{t}^2 u^K + r^2 \del{r}^2 u^K + 2r(1-t)\del{r}\del{t}u^K + \gsl^{\Lambda\Sigma}\nablasl{\Lambda}\nablasl{\Sigma} u^K + 2(t-1)\del{t}u^K = f^K
\end{equation}
where $\nablasl{\Lambda}$ is the Levi-Civita connection of the metric \eqref{gsldef} on $\mathbb{S}^2$.

For the remainder of the article, we will focus on solving the conformal wave equations \eqref{MwaveB} on neighborhoods of spatial infinity of the form
\begin{equation} \label{Mr0def}
M_{r_0} = \bigl\{ (t,r)\in (1,0)\times (0,r_0) \, \bigl| t > 2 - r_0/r\bigr\} \times \mathbb{S}^2 \subset M, \qquad r_0>0,
\end{equation} 
where initial data is prescribed on the spacelike hypersurface
\begin{equation} \label{Sigmar0def}
\Sigma_{r_0} = \{1\}\times (0,r_0)\times \mathbb{S}^2
\end{equation}
that forms the ``top'' of the domain $M_{r_0}$. Noting that
\begin{equation*}
\psi(\Mb_{r_0})=M_{r_0} \AND \psi(\Sigmab_{r_0})=\Sigma_{r_0}
\end{equation*}
by \eqref{Mbr0def} and \eqref{Sigmabr0def}, 
we conclude that any solution of the conformal wave equations on $M_{r_0}$ with initial data prescribed on $\Sigma_{r_0}$
corresponds uniquely to a solution of the semilinear wave equations 
\eqref{Mbwave} on $\Mb_{r_0}$ with initial data prescribed on $\Sigmab_{r_0}$.

\subsection{Expansion formulas for the tensor components $\at^{K\alpha\beta}_{IJ}$}
We now turn to deriving expansion formulas for the tensor components $\at^{K\alpha\beta}_{IJ}$,
defined by \eqref{atdef}, 
that will determine their behavior in the limit $t\searrow 0$. These results will be
essential for transforming the conformal wave equations \eqref{MwaveB} into Fuchsian form, which
will be carried out in the following section.

Now, from \eqref{psibdef} and \eqref{atdef}, we find, after a routine calculation, that 
\begin{align}
\at^{K00}_{IJ}&= 4 t^2  r^2\bt^{K}_{IJ}+  t^3r^2\ct^{K00}_{IJ}, \label{atexp00} \\
\at^{K01}_{IJ} &= -4 t r^3 \bt^{K}_{IJ}+   t^2 r^3 \ct^{K01}_{IJ} , \label{atexp01} \\
\at^{K10}_{IJ}&= -4 t r^3 \bt^{K}_{IJ} + t^2 r^3 \ct^{K10}_{IJ} , \label{atexp10} \\
\at^{K11}_{IJ} &= 4 r^4(1-2t)\bt^{K}_{IJ} + t^2 r^4\ct^{K11}_{IJ},  \label{atexp11}\\
\at^{K0\Lambda}_{IJ} &= -2 t r (\ab^{K0\Lambda}_{IJ}-\ab^{K1\Lambda}_{IJ})\circ\psi^{-1}+ t^2 r(\ab^{K0\Lambda}_{IJ}-3 \ab^{K1\Lambda}_{IJ})\circ\psi^{-1}+ t^3 r\ab^{K1\Lambda}_{IJ}\circ\psi^{-1},  \label{atexp0Lambda}\\
\at^{K\Sigma 0}_{IJ} &= -2 t r (\ab^{K\Sigma 0}_{IJ}-\ab^{K\Lambda 1}_{IJ})\circ\psi^{-1}+
t^2 r(\ab^{K \Sigma 0}_{IJ}-3 \ab^{K\Sigma 1}_{IJ})\circ\psi^{-1}+ t^3 r\ab^{K\Sigma 1}_{IJ}\circ\psi^{-1},\label{atexpLambda0} \\
\at^{K1\Lambda}_{IJ} &=
2 r^2 (\ab^{K0\Lambda}_{IJ}-\ab^{K1\Lambda}_{IJ})\circ\psi^{-1}-2 t r^2 (\ab^{K0\Lambda}_{IJ}
-\ab^{K1\Lambda}_{IJ})\circ\psi^{-1}- t^2 r^2\ab^{K1\Lambda}_{IJ}\circ\psi^{-1} , \label{atexp1Lambda}\\
\at^{K\Sigma 1}_{IJ} &=
2 r^2 (\ab^{K\Sigma 0}_{IJ}-\ab^{K\Sigma 1}_{IJ})\circ\psi^{-1}-2 t r^2 (\ab^{K\Sigma 0}_{IJ}
-\ab^{K\Sigma 1}_{IJ})\circ\psi^{-1}-  t^2 r^2\ab^{K\Sigma 1}_{IJ}\circ\psi^{-1}  \label{atexpLambda1}
\intertext{and}
\at^{K \Sigma \Lambda}_{IJ} & = \ab^{K \Sigma \Lambda}_{IJ} \circ \psi^{-1}, \label{atexpSigmaLambda}
\end{align}
where
\begin{align}
\bt^{K}_{IJ} &= (\ab^{K00}_{IJ}-\ab^{K01}_{IJ}-\ab^{K10}_{IJ}+\ab^{K11}_{IJ})\circ\psi^{-1}, 
\label{btdef} \\
\ct^{K00}_{IJ} &=-4 \left((\ab^{K00}_{IJ}-2
   \ab^{K01}_{IJ}-2 \ab^{K10}_{IJ}+3 \ab^{K11}_{IJ})\right)\circ\psi^{-1}+ t (\ab^{K00}_{IJ}-5 \ab^{K01}_{IJ}-5
   \ab^{K10}_{IJ}+13 \ab^{K11}_{IJ})\circ\psi^{-1}\notag \\
   & \hspace{4.0cm}+t^2 (\ab^{K01}_{IJ}+\ab^{K10}_{IJ}-6 \ab^{K11}_{IJ})\circ\psi^{-1}+t^3 \ab^{K11}_{IJ}\circ\psi^{-1}
   \label{ct00def} \\
\ct^{K01}_{IJ} &=   2(3
   \ab^{K00}_{IJ}-3 \ab^{K01}_{IJ}-5 \ab^{K10}_{IJ}+5 \ab^{K11}_{IJ})\circ\psi^{-1}-2 t \left((\ab^{K00}_{IJ}-2
   \ab^{K01}_{IJ}-4 \ab^{K10}_{IJ}+5 \ab^{K11}_{IJ})\right)\circ\psi^{-1} \notag \\
   &\hspace{4.0cm}- t^2 (\ab^{K01}_{IJ}+2 \ab^{K10}_{IJ}-5
   \ab^{K11}_{IJ})\circ\psi^{-1}- t^3 \ab^{K11}_{IJ}\circ\psi^{-1} , \label{ct01def} \\
\ct^{K10}_{IJ} &=2 (3
   \ab^{K00}_{IJ}-5 \ab^{K01}_{IJ}-3 \ab^{K10}_{IJ}+5 \ab^{K11}_{IJ})\circ\psi^{-1}-2 t \left((\ab^{K00}_{IJ}-4
   \ab^{K01}_{IJ}-2 \ab^{K10}_{IJ}+5 \ab^{K11}_{IJ})\right)\circ\psi^{-1} \notag \\
   &\hspace{4.0cm}- t^2 (2 \ab^{K01}_{IJ}+\ab^{K10}_{IJ}-5
   \ab^{K11}_{IJ})\circ\psi^{-1}-  t^3 \ab^{K11}_{IJ}\circ\psi^{-1} \label{ct10def}
\intertext{and}
\ct^{K11}_{IJ}&= 2  (2 \ab^{K00}_{IJ}-3
   \ab^{K01}_{IJ}-3 \ab^{K10}_{IJ}+4 \ab^{K11}_{IJ})\circ\psi^{-1} \notag \\
   &\hspace{4.0cm}+2 t (\ab^{K01}_{IJ}+\ab^{K10}_{IJ}-2
   \ab^{K11}_{IJ})\circ\psi^{-1}+t^2 \ab^{K11}_{IJ}\circ\psi^{-1}.  \label{ct11def} 
\end{align}
We further observe from \eqref{psibdef}, \eqref{tr2tbrb}, \eqref{abexp}-\eqref{bbIJKformula}  and \eqref{btdef} that
\begin{align}
\ab^{K \pc \qc}_{IJ}\circ \psi^{-1} &= \cb_{IJ}^{K\pc \qc}+
t r (2-t)\db_{IJ}^{K\pc \qc}+  t^2 r^2 (2-t)^2 \eb_{IJ}^{K\pc \qc}, \label{ab-exp-pc-qc} \\
\ab^{K \alpha \Lambda}_{IJ}\circ \psi^{-1} &= t r (2-t)\db_{IJ}^{K\alpha \Lambda}+  t^2 r^2 (2-t)^2 \eb_{IJ}^{K\alpha \Lambda} \label{ab-exp-alpha-Lambda}, \\
\ab^{K \Sigma \beta}_{IJ}\circ \psi^{-1} &=t r (2-t)\db_{IJ}^{K\Sigma\beta}+  t^2 r^2 (2-t)^2 \eb_{IJ}^{K\Sigma \beta}
\label{ab-exp-Sigma-beta} \\
\intertext{and}
\bt^{K}_{IJ} &= \bb^{K}_{IJ}. \label{bt=bb}
\end{align}

\section{A Fuchsian formulation}

\subsection{First order variables\label{firstorder:sec}}
We now begin the process of transforming the conformal wave equations \eqref{MwaveB} into a Fuchsian form. The transformation starts by
expressing the wave equation in first order form through the introduction of the variables
\begin{equation}\label{U0def}
U_0^K = t\del{t}u^K,  \quad
U_1^K = t^{\frac{1}{2}}r\del{r}u^K, \quad U_{\Lambda}^K =t^{\frac{1}{2}}\nablasl{\Lambda} u^K 
\AND
U_4^K = t^{\frac{1}{2}}u^K. 
\end{equation}
A short calculation then shows that \eqref{MwaveB}, when expressed in terms of these variables, becomes
\begin{align}
(2-t)\del{t} U^K_0 - \frac{2(1-t)}{t}r\del{r}U^K_0 - \frac{1}{t^{\frac{1}{2}}}r\del{r} U_1^K - 
\frac{1}{t^{\frac{1}{2}}}\gsl^{\Lambda\Sigma}\nablasl{\Lambda}U^K_{\Sigma} &= 
-\frac{1}{t^{\frac{1}{2}}}U^K_1
+U_0^K-f^K,  \label{MwaveC}
\end{align}
while the evolution equations for the variables $U_1^K$, $U^K_\Lambda$ and $U_4^K$ are easily computed to be
\begin{equation} \label{MwaveD.1}
\del{t}U_1^K = \frac{1}{t^{\frac{1}{2}}}r\del{r} U_0^K + \frac{1}{2 t}U^K_1, \quad
\del{t}U_\Lambda^K =  \frac{1}{t^{\frac{1}{2}}}\nablasl{\Lambda}U^K_0 + \frac{1}{2 t}U^K_\Lambda  
\AND
\del{t}U_4^K = \frac{1}{2 t}U^K_4+\frac{1}{t^{\frac{1}{2}}}U_0^K, 
\end{equation}
respectively. It is worthwhile noting that system \eqref{MwaveC}-\eqref{MwaveD.1} is in symmetric hyperbolic form.

To proceed, we use the first order variables \eqref{U0def} to write $\nabla_\mu u^I$ as
\begin{equation*}
	\nabla_\mu u^I = t^{-\frac{1}{2}}\bigl(t^{-\frac{1}{2}}U^I_0\delta_\mu^0+ r^{-1}U^I_1\delta^1_\mu
+U^I_\Lambda\delta^\Lambda_\mu\bigr).
\end{equation*}
Using this, we then observe that the three main groups of terms from \eqref{fKdef} can be expressed in terms of the 
first order variables as
\begin{align*}
- &\frac{1}{rt(2-t)}\at_{IJ}^{K\mu\nu}\nabla_\mu u^I\nabla_\nu u^J 
=  -\frac{1}{2-t}r^{-1}t^{-2}\frac{1}{t}\biggl[
\biggl(\at^{K00}_{(IJ)}U^I_0 U^J_0 +\frac{t^{\frac{1}{2}}}{r}\bigl(\at^{K01}_{IJ}+\at^{K10}_{JI}\bigr)U^I_0 U^J_1\notag \\
&
+\frac{t}{r^2}\at^{K11}_{(IJ)}U^I_1 U^J_1\biggr)+ t^{\frac{1}{2}}\bigl(\at^{K0\Lambda}_{JI} +\at^{K\Lambda 0}_{IJ}\bigr)U^I_\Lambda U^J_0 
+ \frac{t}{r}\bigl(\at^{K1\Lambda}_{JI} + \at^{K\Lambda 1}_{IJ}\bigr)U^I_\Lambda U^J_1
+ t\at^{K\Lambda\Sigma}_{IJ}U^I_\Lambda U^J_{\Sigma}\biggr], 
%\label{MwaveE.1}
\end{align*}
\begin{align*}
-&\frac{1}{(rt(2-t))^2}\at^{K\mu\nu}_{IJ}
\bigl(\nabla_\mu(rt(2-t))u^I\nabla_\nu u^J+ \nabla_\mu u^I \nabla_\nu(rt(2-t))u^J\bigr)
\notag \\
&= -\frac{1}{(2-t)^2}r^{-2}t^{-2}\frac{1}{t}
\biggl[\frac{1}{t^{\frac{1}{2}}}\biggl(\bigl(2r(1-t)\at^{K00}_{IJ}+ t(2-t)\at^{K10}_{IJ}\bigr)U^I_4 U^J_0
\notag \\
&\qquad +2t^{\frac{1}{2}}r(1-t)\at^{K0\Sigma}_{IJ}U^I_4U^J_\Sigma
+\biggl(\frac{t^{\frac{3}{2}}(2-t)\at^{K11}_{IJ}}{r}+ 2t^{\frac{1}{2}}(1-t)\at^{K01}_{IJ}\biggr)U^I_4 U^J_1 + t^{\frac{3}{2}}(2-t)\at^{K1\Sigma}_{IJ}U^I_4 U^J_\Sigma\biggr)\notag\\
&\qquad+\frac{1}{t^{\frac{1}{2}}}\biggl( \bigl(2r(1-t)\at^{K00}_{IJ}+ t(2-t)\at^{K01}_{IJ}\bigr)U^I_0 U^J_4 
+2t^{\frac{1}{2}}r(1-t)\at^{K\Lambda 0}_{IJ}U^I_\Lambda U^J_4 \notag \\
&\qquad +\biggl(\frac{t^{\frac{3}{2}}(2-t)\at^{K11}_{IJ}}{r}+ 2t^{\frac{1}{2}}(1-t)\at^{K10}_{IJ}\biggr)U^I_1 U^J_4 + t^{\frac{3}{2}}(2-t)\at^{K\Lambda 1}_{IJ}U^I_\Lambda U^J_4\biggr)
\biggr]
%\label{MwaveE.2}
\end{align*}
and
\begin{align*}
-\frac{1}{(rt(2-t))^3}\at^{K\mu\nu}_{IJ}\nabla_\mu(rt(2-t))&\nabla_\nu(rt(2-t))u^I u^J = -\frac{1}{(2-t)^3}r^{-3}t^{-3}\frac{1}{t}
\bigl[4r^2(1-t)^2\at^{K00}_{(IJ)}\notag \\
&\qquad + 2tr(1-t)(2-t)\bigl(\at^{K01}_{(IJ)}+\at^{K10}_{(IJ)}\bigr)+t^2(2-t)^2\at^{K11}_{(IJ)}\bigr]U^I_4 U^J_4. %\label{MwaveE.3}
\end{align*}
With the help of these results, it is then not difficult to verify, using \eqref{bbIJKformula}, \eqref{atexp00}-\eqref{atexpSigmaLambda}
and \eqref{ab-exp-pc-qc}-\eqref{bt=bb}, that the nonlinear term \eqref{fKdef}becomes
\begin{align}
-f^K &= -\frac{1}{t} 2 r \bb^K_{IJ}V^I_0 V^J_0 + \frac{1}{t}\biggl[ r \fc^{K00}_{IJ} t^{\frac{1}{2}}U^I_0 t^{\frac{1}{2}}U^J_0
+  r \fc^{K01}_{IJ} t^{\frac{1}{2}}U^I_0 U^J_1  \notag \\
&\qquad+ r \fc^{K11}_{IJ}U^I_1 U^J_1
 + \fc^{K0\Lambda}_{IJ}t^{\frac{1}{2}}U^I_0 U^J_\Lambda
+ \fc^{K1\Lambda}_{IJ} U^I_1 U^J_\Lambda + \fc^{K\Sigma\Lambda}_{IJ} U^I_\Sigma U^J_\Lambda\notag \\
&\qquad + r\gc^{K0}_{IJ}t^{\frac{1}{2}}U^I_0 U^J_4 + r\gc^{K1}_{IJ}U^I_1 U^J_4 
 + \gc^{K\Lambda}_{IJ} t^{\frac{1}{2}}U^I_\Lambda U^J_4 + r\hc^{K}_{IJ}U^I_4 U^J_4 \label{fKexp}
\end{align}
when written in terms of the first order variables, 
where $\{\fc^{K\pc\qc}_{IJ}(t,r)$, $\gc^{K\pc}_{IJ}(t,r)$, $\hc^{K}_{IJ}(t,r)\}$, $\{\fc^{K\pc\Lambda}_{IJ}(t,r)$, $\gc^{K\Lambda}_{IJ}(t,r)\}$
and $\{f^{K\Sigma\Lambda}_{IJ}(t,r)\}$ are collections of smooth scalar, vector, and (2,0)-tensor fields, respectively, on $\mathbb{S}^2$ that
depend smoothly on $(t,r)\in \Rbb \times \Rbb$, and 
we have set 
\begin{equation}\label{V0Kdef}
V_0^K = U^K_0-\frac{1}{t^{\frac{1}{2}}}U^K_1.
\end{equation}

The expansion \eqref{fKexp} motivates us to replace the first order variable $U_0^K$ with $V_0^K$. Doing so, we
see via a routine computation involving \eqref{MwaveC} and \eqref{MwaveD.1} that $V_0^K$ evolves according
to 
\begin{equation} \label{MwaveF}
(2-t)\del{t}V_0^K +r\del{r}V_0^K - 
\frac{1}{t^{\frac{1}{2}}}\gsl^{\Lambda\Sigma}\nablasl{\Lambda}U^K_{\Sigma}=V_0^K - f^K.
\end{equation}
One difficulty with this change of variables is  the system of evolution equations \eqref{MwaveD.1} and \eqref{MwaveF} for the first order variables $V^K_0$ , $U^K_1$, $U^K_\Lambda$ and $U^K_4$
is no longer symmetric hyperbolic. To restore the symmetry, we use the identity $\nablasl{\Lambda} U_1=r \del{r}U_\Lambda$
to write \eqref{MwaveD.1} as
\begin{align}
\del{t}U^K_1&= \frac{1}{t}r\del{r}U_1^K + \frac{1}{t^{\frac{1}{2}}}r\del{r}V_0^K+\frac{1}{2 t}U_1^K,
\label{MwaveF.1}\\
\del{t}U^K_\Lambda & = -\frac{\qtt}{t} r\del{r}U^K_\Lambda + \frac{1}{t^{\frac{1}{2}}}\nablasl{\Lambda}U^K_0
+\frac{\qtt+1}{t}\nablasl{\Lambda}U^K_1 + \frac{1}{2 t}U^K_\Lambda, \label{MwaveF.2} \\
\del{t}U^K_4 &= \frac{1}{2t}U_4 + \frac{1}{t}U_1^K
+ \frac{1}{t^{\frac{1}{2}}}V_0^K ,  \label{MwaveF.3}
\end{align}
where
\begin{equation}\label{qttdef} 
\qtt=\frac{-1+2 t^2 - t^3}{1+ 4 t- 4t^2+t^3}.
\end{equation}
We then define new first order variables by setting
\begin{equation} \label{V1def}
V_1^K = 2 U_1^K + (2-t)t^{\frac{1}{2}}V_0^K,  \quad
V^K_\Lambda = \ptt U^K_\Lambda 
\AND
V^K_4  = U^K_4, 
\end{equation}
where
\begin{equation} \label{pdef}
\ptt=\sqrt{\frac{1+4t-4t^2+t^3}{2-t}},
\end{equation}
and observe that the evolution equations \eqref{MwaveF}-\eqref{MwaveF.3} can be expressed
in terms of the variables $V_0^K$, $V^K_1$, $V^K_\Lambda$ and $V^K_4$ in the following symmetric hyperbolic form:
\begin{equation}\label{MwaveG}
B^{0}\del{t}V^K+\frac{1}{t}B^{1} r\del{r}V^K + \frac{1}{t^{\frac{1}{2}}} B^{\Sigma} \nablasl{\Sigma}V^K
= \frac{1}{t}\Bc \Pbb V^K + \Cc V^K+ F^K
\end{equation}
where
\begin{align} 
V^K&= (V^K_\Ic) = \begin{pmatrix}V^K_0 & V^K_1 & V^K_\Lambda & V^K_4\end{pmatrix}^{\tr}, \label{Vdef}\\
B^0&= \begin{pmatrix}2-t & 0 & 0 & 0\\ 0 & 2-t & 0 & 0 \\ 0 & 0 &(2-t)\delta_{\Omega}^{\Lambda} & 0\\
0 & 0 & 0 & 1 \end{pmatrix}, \label{B0def}\\
B^1&= \begin{pmatrix} t & 0 & 0 & 0\\ 0 & -(2-t) & 0 & 0 \\ 0 & 0 & (2-t)\qtt \delta_{\Omega}^{\Lambda} & 0\\
0 & 0 & 0 & 0 \end{pmatrix}, \label{B1def}\\
B^\Sigma&= \begin{pmatrix} 0& 0 & -\frac{1}{\ptt}\gsl^{\Sigma\Lambda}  & 0\\ 0 & 0 &  -\frac{(2-t)t^{\frac{1}{2}}}{\ptt}\gsl^{\Sigma\Lambda} & 0 \\ -\frac{1}{\ptt} \delta_{\Omega}^{\Sigma} &  -\frac{(2-t)t^{\frac{1}{2}}}{\ptt} \delta_{\Omega}^{\Sigma} & 0 & 0\\
0 & 0 & 0 & 0 \end{pmatrix} \label{BSigmadef},\\
\Bc&= \begin{pmatrix} 2 & 0 &0  & 0\\0 & \frac{2-t}{2} & 0 & 0 \\ 0 & 0 & \frac{2-t}{2}\gsl_{\Omega}^{\Sigma} & 0\\
0 & \frac{1}{2} & 0 &\frac{1}{2} \end{pmatrix}, \label{Bcdef}\\
\Cc &= \begin{pmatrix}
 1 & 0 & 0 & 0 \\
 0 & 0  & 0 & 0\\
 0 & 0 & \frac{9-16 t+10 t^2- 2 t^3}{2(1+4 t -4 t^2 + t^3)} \delta_{\Omega}^{\Lambda}  & 0\\
   \frac{1}{2}t^{\frac{1}{2}} & 0 & 0 & 0
   \end{pmatrix}, \label{Ccdef}\\
\Pbb &= \begin{pmatrix}
 0 & 0 & 0 & 0 \\
 0 & 1  & 0 & 0\\
 0 & 0 & \delta_\Sigma^\Lambda & 0\\
0 & 0 & 0 & 1
   \end{pmatrix} \label{Pbbdef}
\intertext{and}
F^K & = \begin{pmatrix}-f^K &
-(2-t)t^{\frac{1}{2}}f^K & 0 & 0 \end{pmatrix}^{\tr}. \label{Fdef}
\end{align}

Now, from the definitions  \eqref{B0def}, \eqref{B1def}, \eqref{Bcdef} and \eqref{Pbbdef}, it is not difficult to verify that $\Pbb$ is a covariantly constant, time-independent, symmetric projection operator that commutes
with $B^0$, $B^1$ and $\Bc$, that is,
\begin{equation} \label{Pbbprops}
\Pbb^2 = \Pbb, \quad \Pbb^{\tr}=\Pbb, \quad \quad \del{t}\Pbb=0, \quad \del{r}\Pbb =0, \AND \nablasl{\Lambda}\Pbb = 0
\end{equation} 
and
\begin{equation} \label{Pbbcom}
[B^0,\Pbb]=[B^1, \Pbb] = [\Bc,\Pbb]=0,
\end{equation}
where the symmetry is with respect to the inner-product
\begin{equation} \label{hdef}
h(Y,Z) = \delta^{\pc\qc}Y_\pc Z_\qc + \gsl^{\Sigma\Lambda}Y_\Lambda Z_\Sigma + Y_4 Z_4.
\end{equation}
Furthermore, it is also not difficult to verify that $B^0$ and $B^1$ and $B^\Sigma \eta_\Sigma$ are symmetric with respect to \eqref{hdef}
and that $B^0$ satisfies
\begin{equation} \label{B0lowbnd}
h(Y,Y)\leq h(Y,B^0Y) 
\end{equation} 
for all $Y=(Y_\Ic)$ and $0<t\leq 1$, which in particular, implies that the system \eqref{MwaveG} is symmetric hyperbolic.

Using \eqref{Bcdef} and \eqref{hdef}, we observe, with the help of Young's inequality (i.e. $|ab|\leq \frac{\ep}{2}a^2+\frac{1}{2\ep}b^2$), that
\begin{align*}
h(Y,\Bc Y) &= 2 Y_0^2+\frac{2-t}{2} Y_1^2 + Y_1 Y_4 +\frac{2-t}{2}\gsl^{\Lambda\Sigma}Y_\Lambda Y_\Sigma +\frac{1}{2} Y_4^2\\
&\geq 2 Y_0^2+\frac{2-t-\ep}{2} Y_1^2 +\frac{2-t}{2}\gsl^{\Lambda\Sigma}Y_\Lambda Y_\Sigma +\frac{1}{2}\biggl(1-\frac{1}{\ep}\biggr) Y_4^2.
\end{align*}
Choosing $\ep = \frac{1}{2}\bigl(1-t-\sqrt{5-2t+t^2}\bigr)$, we then have
\begin{align*}
h(Y,\Bc Y) &\geq 2 Y_0^2+\frac{1}{4}\bigl(3-t+\sqrt{5-2 t+t^2}\bigr) Y_1^2  +\frac{2-t}{2}\gsl^{\Lambda\Sigma}Y_\Lambda Y_\Sigma +
\frac{1}{4}\bigl(3-t+\sqrt{5-2 t+t^2}\bigr) Y_4^2 \\
& \geq 2 Y_0^2 + \frac{2-t}{2}\bigl(Y_1^2 +\gsl^{\Lambda\Sigma}Y_\Lambda Y_\Sigma  + Y^2_4 \bigr) \\
& \geq \frac{1}{2}\Bigl((2-t)Y_0^2 + (2-t)Y_1^2 +(2-t)\gsl^{\Lambda\Sigma}Y_\Lambda Y_\Sigma  + Y^2_4 \Bigr),
\end{align*}
which together with \eqref{B0def} and \eqref{hdef} allows us to conclude that
\begin{equation} \label{B0Bcbnd}
h(Y,B^0 Y) \leq 2 h(Y,\Bc Y)
\end{equation}
for all $Y=(Y_\Ic)$ and $0<t\leq 1$.

Next, from \eqref{V0Kdef} and \eqref{V1def}, we get
\begin{equation}
t^{\frac{1}{2}}U_0^K= \frac{1}{2}\bigl(V_1^K + t V_0^K \bigr), \quad U_1^K=\frac{1}{2}\bigl(V_1^K-(2-t)t^{\frac{1}{2}}V_0^K\bigr), \quad U_\Lambda^K=\frac{1}{\ptt}V^K_\Lambda \AND U_4^K=V_4^K. \label{VtoU}
\end{equation}
Using these along with \eqref{fKexp} and \eqref{Pbbdef} allows us to expand \eqref{Fdef} as
\begin{equation} \label{FKexp}
F^K = -\frac{2}{t}\bb^{K}_{IJ}rV^I_0 V_0^J \ev_0 + G^K
\end{equation}
where
\begin{equation} \label{FcKexp}
G^K =%\frac{1}{t^{\frac{1}{2}}} G_0^K(t^{\frac{1}{2}},t,r,\Pbb^\perp\Vv,\Pbb V)+ 
 %\frac{1}{t^{\frac{1}{2}}} G_1^K(t^{\frac{1}{2}},t,r,r\Vv,V)
 G_0^K(t^{\frac{1}{2}},t,r,V,V)+
 \frac{1}{t^{\frac{1}{2}}} G_1^K(t^{\frac{1}{2}},t,r,V,\Pbb V)
+ \frac{1}{t} G_2^K(t^{\frac{1}{2}},t,r,\Pbb V, \Pbb V),
\end{equation}
\begin{equation}\label{ev0def}
\ev_0 = ( \delta_\Ic^0) = \begin{pmatrix} 1 & 0 & 0 & 0 \end{pmatrix}^{\tr}, 
\end{equation}
%\begin{equation} \label{Pbbperpdef}
%\Pbb^\perp = \id - \Pbb
%\end{equation}
%is the complementary projection, 
\begin{equation} \label{Vvecdef}
V= (V^I)=(V^I_\Ic),
\end{equation}
and  the $G^K_a(\tau,t,r,Y,Z)$, %$a=0,1,2$, 
$a=0,1,2$, are smooth bilinear maps with $G^K_2$ satisfying
\begin{equation}\label{PbbGK=0}
\Pbb G^K_2 = 0.
\end{equation} 

\begin{rem}
Here, we are using the term \textit{smooth bilinear map} to mean a map of the form
\begin{equation*}
H^K(\tau,t,r,Y,Z) = H^{K\pc\qc}_{IJ}(\tau,t,r)Y^I_\pc Z^J_\pc +  H^{K\pc\Lambda}_{IJ}(\tau,t,r)Y^I_\pc Z^J_\Lambda+ 
H^{K\Sigma\Lambda}_{IJ}(\tau,t,r) Y^I_\Sigma Z^J_\Lambda
\end{equation*}
where $H^{K\pc\qc}_{IJ}(\tau,t,r)$, $H^{K\pc\qc}_{IJ}(\tau,t,r)$, and $H^{K\Sigma\Lambda}_{IJ}(\tau,t,r)$ are collections of
smooth scalar, vector, and (2,0)-tensor fields on $\mathbb{S}^2$ that depend smoothly on the parameters $(\tau,t,r)\in \Rbb\times \Rbb \times \Rbb$.
\end{rem}

For the subsequent analysis, it will be advantageous to introduce a change of radial coordinate via
\begin{equation} \label{rhodef}
r= \rho^m, \quad m\in \Zbb_{\geq 1}.
\end{equation}
Using the transformation law $r\del{r} = r \frac{d\rho}{dr}\del{\rho}= \frac{\rho}{m}\del{\rho}$, we can express the system
\eqref{MwaveG} as
\begin{equation}\label{MwaveH}
B^{0}\del{t}V^K+\frac{1}{t} \frac{\rho}{m}B^{1}\del{\rho}V^K + \frac{1}{t^{\frac{1}{2}}} B^\Sigma \nablasl{\Sigma}V^K
= \frac{1}{t}\Bc \Pbb V^K + \Cc V^K+ F^K
\end{equation}
where now
\begin{equation} \label{FKexprho}
F^K = -\frac{2}{t}\bb^{K}_{IJ}\rho^m V^I_0 V_0^J \ev_0 + G^K
\end{equation}
and
\begin{equation} \label{FcKexprho}
G^K =  %\frac{1}{t^{\frac{1}{2}}} G_0^K(t^{\frac{1}{2}},t,\rho^{m},\Pbb^\perp\Vv,\Pbb V)+  
%\frac{1}{t^{\frac{1}{2}}} G_1^K(t^{\frac{1}{2}},t,\rho^{m},\rho^{m}V,V)
 G_0^K(t^{\frac{1}{2}},t,\rho^{m},V,V)+\frac{1}{t^{\frac{1}{2}}} G_1^K(t^{\frac{1}{2}},t,\rho^{m},V,\Pbb V)
+ \frac{1}{t} G_2^K(t^{\frac{1}{2}},t,\rho^{m},\Pbb V, \Pbb V).
\end{equation}
It is also clear that the neighborhood of infinity $M_{r_0}$ and the initial data hypersurface 
$\Sigma_{r_0}$, see \eqref{Mr0def} and \eqref{Sigmar0def},  can be expressed in terms of $\rho$ as 
\begin{equation} \label{Mr02rho}
M_{r_0} = \bigl\{ (t,\rho)\in (1,0)\times (0,\rho_0) \, \bigl| t > 2 - \rho^m_0/\rho^m \bigr\} \times \mathbb{S}^2, \qquad 
\rho_0= (r_0)^{\frac{1}{m}},
\end{equation} 
and
\begin{equation} \label{Sigmar02rho}
\Sigma_{r_0} = \{1\}\times (0,\rho_0)\times \mathbb{S}^2,
\end{equation}
respectively.

\subsection{The extended system\label{extendedsec}}
Rather than solving \eqref{MwaveH} on $M_{r_0}$, we will instead solve an extended version of this system
on the extended spacetime $(0,1)\times \Sc$
where 
\begin{equation*}
\Sc = \Tbb \times \mathbb{S}^2
\end{equation*}
and $\Tbb$ is the 1-dimensional torus obtained from identifying the end points of the interval $[-3\rho_0,3\rho_0]$.
Initial data will be prescribed on the hypersurface $\{1\}\times \Sc$.

To define the extended system, we let $\hat{\chi}(\rho)$ denote a smooth cut-off function satisfying
$\hat{\chi} \geq 0$, $\hat{\chi}|_{[-1,1]} = 1$ and $\supp(\hat{\chi}) \subset (-2,2)$,
and use it to define the smooth cut-off function
\begin{equation*}%\label{chirho0def}
\chi(\rho) = \hat{\chi}\bigl(\rho/\rho_0\bigr)
\end{equation*}
on $\Tbb$, which is easily seen to satisfy
$\chi \geq 0$, $\chi|_{[-\rho_0,\rho_0]} = 1$ and $\supp(\chi) \subset (-2\rho_0,2\rho_0)$.
With the help of this cut-off function, we then define the \textit{extended system} by 
\begin{equation}\label{MwaveI}
B^{0}\del{t}V^K+\frac{1}{t} \frac{\chi\rho}{m}B^{1}\del{\rho}V^K + \frac{1}{t^{\frac{1}{2}}} B^\Sigma \nablasl{\Sigma}V^K
= \frac{1}{t}\Bc \Pbb V^K + \Cc V^K+ \Fc^K
\end{equation}
where
\begin{gather} 
\Fc^K = \frac{1}{t}Q^K \ev_0 + \Gc^K, \label{FcKdef} \\
Q^K = -2\bb^{K}_{IJ}\chi(\rho)\rho^m V^I_0 V_0^J, \label{QKdef} \\
\Gc^K = %\frac{1}{t^{\frac{1}{2}}}\Gc_0 +  
\Gc_0+\frac{1}{t^{\frac{1}{2}}}\Gc_1 +
 \frac{1}{t}\Gc_2, \label{Gcdef}\\
%\Gc_0 =   G_0^K(t^{\frac{1}{2}},t,\chi(\rho)\rho^m,\chi(\rho)\rho^m \Pbb^\perp V,\Pbb V), \label{GcK0def} \\
%\Gc_1^K =  G_1^K(t^{\frac{1}{2}},t,\chi(\rho)\rho^m,\chi(\rho)\rho^m V,V) \label{GcK1def}
\Gc_0^K =   G_0^K(t^{\frac{1}{2}},t,\chi(\rho)\rho^m,V,V), \label{GcK0def} \\
\Gc_1^K =  G_1^K(t^{\frac{1}{2}},t,\chi(\rho)\rho^m,V,\Pbb V), \label{GcK1def}\\
\Gc_2^K =  G_2^K(t^{\frac{1}{2}},t,\chi(\rho)\rho^m,\Pbb V,\Pbb V) \label{GcK2def}
\intertext{and}
\Pbb \Gc_2^K = 0. \label{PbbGc2=0}
\end{gather}

By definition, see \eqref{Vdef}, the fields $V^K$ are time-dependent sections of the vector bundle
\begin{equation*}
\Vbb = \bigcup_{y\in \Sc} \Vbb_{y}
\end{equation*}
over $\Sc$ with fibers $\Vbb_y = \Rbb\times \Rbb \times \text{T}^*_{\text{pr}(y)}\mathbb{S}^2\times \Rbb$
where $\text{pr} :  \Sc = \Tbb \times \mathbb{S}^2 \longrightarrow \mathbb{S}^2$ is the canonical projection.
We further note that \eqref{hdef} defines an inner-product on  $\Vbb$, and 
recall that $B^0$, $B^1$ and $B^\Sigma\xi_\Sigma$
are symmetric with respect to this inner-product. The symmetry of these operators together with the lower
bound  \eqref{B0lowbnd} for $B^0$ imply that the extended system \eqref{MwaveI} is symmetric hyperbolic, a fact that
will be essential to our arguments below.

Noting from the definition \eqref{Mr02rho} that the boundary of the region $M_{r_0}$  can be decomposed as
\begin{equation*}
\del{}M_{r_0} = \Sigma_{r_0}\cup \Sigma^+_{r_0} \cup \Gamma^{-} \cup \Gamma^{+}_{r_0}
\end{equation*}
where
\begin{gather*}
\Gamma^{-} = [0,1] \times \{0\} \times \mathbb{S}^2,\quad
\Gamma^{+}_{r_0} = \biggl\{ \, (t,r) \in [0,1]\times (0,\rho_0] \: \biggl|\: t = 2- \frac{\rho_0^m}{\rho^m} \, \biggr\} \times \mathbb{S}^2
\AND
\Sigma^+_{r_0} = \{0\}\times \Bigl(0,\frac{\rho_0}{2^{\frac{1}{m}}}\Bigr)\times \mathbb{S}^2, 
\end{gather*} 
we find that
$n^{-}= -d\rho$ and $n^{+} = -dt + m\frac{\rho_0^m}{\rho^{m+1}} d\rho$
define outward pointing co-normals to $\Gamma^{-}$ and $\Gamma^{+}_{r_0}$, respectively. Furthermore, we have from \eqref{B0def}-\eqref{BSigmadef} that
\begin{gather}
\Bigl(n^{-}_{0} B^0 +n^{-}_1 \frac{\chi\rho}{m} B^1 + n^{-}_\Sigma B^\Sigma\Bigr) \Bigl|_{\Gamma^{-}} = 0 \label{wspacelike1}
\intertext{and}
 \Bigl(n^{+}_{0} B^0 +n^{+}_1 \frac{\chi\rho}{m} B^1 + n^{+}_\Sigma B^\Sigma\Bigr)\Bigl|_{\Gamma^{+}_{r_0}}  = \begin{pmatrix}
 -(1-t)(2-t)& 0 & 0 & 0\\
0 &-(2-t)(3-t) & 0 & 0 \\
0 & 0 & -(2-t)(1-\qtt(2-t))\delta^\Lambda_\Omega & 0\\
0 & 0 & 0 & 0
\end{pmatrix}, \label{wspacelike2}
\end{gather}
where in deriving this we have used the fact that $2-t=\frac{\rho_0^m}{\rho^m}$ on $\Gamma^+_{r_0}$. By \eqref{qttdef}, we have that $1-\qtt(2-t)$ satisfies
$1<1-\qtt(2-t) < 3$ for $0<t<1$. From this inequality, \eqref{hdef}, \eqref{wspacelike1} and \eqref{wspacelike2},
we deduce that
\begin{equation*}
h\Bigl(Y,\Bigl(n^{-}_{0} B^0 +n^{-}_1 \frac{\chi\rho}{m} B^1 + n^{-}_\Sigma B^\Sigma\Bigr) \Bigl|_{\Gamma^{-}}Y\Bigr) \leq 0 \AND
\quad h\Bigl(Y,\Bigl(n^{+}_{0} B^0 +n^{+}_1 \frac{\chi\rho}{m} B^1 + n^{+}_\Sigma B^\Sigma\Bigr)\Bigl|_{\Gamma^{+}_{r_0}}Y\bigr) \leq 0
\end{equation*}
for all $Y=(Y_\Ic)$. Consequently, by definition, see \cite[\S 4.3]{Lax:2006}, the surfaces $\Gamma^{-}$ and $\Gamma^+_{r_0}$ are weakly spacelike, and it follows that any solution of the extended system \eqref{MwaveI} on the extended spacetime $(0,1)\times\Sc$ will yield
by restriction a solution of the system 
\eqref{MwaveH} on the region \eqref{Mr02rho} that is uniquely determined by the restriction of the initial data to \eqref{Sigmar02rho}. From this property and the above arguments, we conclude that the existence 
of solutions to the conformal wave equations \eqref{MwaveB} on $M_{r_0}$ can
be obtained from solving the initial value problem 
\begin{align} 
B^{0}\del{t}V^K+\frac{1}{t}\frac{\chi\rho}{m}B^{1} \del{\rho}V^K + \frac{1}{t^{\frac{1}{2}}} B^\Sigma \nablasl{\Sigma}V^K
&= \frac{1}{t}\Bc \Pbb V^K + \Cc V^K+ \Fc^K  && \text{in $(0,1) \times \Sc$,}  \label{MwaveJ.1}\\
V^K &= \mathring{V}{}^K && \text{in $\{1\} \times \Sc$,}  \label{MwaveJ.2}
\end{align} 
for initial data $\mathring{V}{}^K=(\mathring{V}_\Ic^K)$ satisfying the constraints 
\begin{equation} \label{idataconstraint}
\nablasl{\Lambda}\mathring{V}{}^K_4 = \frac{1}{\sqrt{2}}\mathring{V}{}^K_\Lambda \AND 
\frac{\rho}{m}\del{\rho}\mathring{V}{}^K_4= \frac{1}{2}\bigl( \mathring{V}{}^K_1-\mathring{V}{}^K_0\bigr)
\quad \text{in $\Sigma_{r_0}$.}
\end{equation}
Moreover, solutions to \eqref{MwaveB} generated this way are independent of the particular form of the initial data 
$\mathring{V}$ on $(\{1\}\times 
\Sc)\setminus \Sigma_{r_0} $ and are determined from solutions of the IVP \eqref{MwaveJ.1}-\eqref{MwaveJ.2} via
\begin{equation} \label{VK2uK}
u^K(t,r,\theta,\phi) = \frac{1}{t^{\frac{1}{2}}}V_4^K(t,r^{\frac{1}{m}},\theta,\phi).
\end{equation} 
Finally, solutions to the semilinear wave equations \eqref{Mbwave} on $\Mb_{r_0}$ can then be obtained from \eqref{VK2uK}
using \eqref{utdef} and \eqref{utK2uK}, which yield the explicit formula 
\begin{equation} \label{VK2ubK}
\ub^K(\tb,\rb,\theta,\phi) = \frac{\rb}{\rb^2-\tb^2}\biggl(1-\frac{\tb}{\rb}\biggr)^{\frac{1}{2}}\biggl(1+\frac{\tb}{\rb}\biggr)
V_4^K\biggl(1-\frac{\tb}{\rb},\biggl(\frac{\rb}{\rb^2-\tb^2}\biggr)^{\frac{1}{m}},\theta,\phi\biggr).
\end{equation}

\subsubsection{Initial data transformations\label{idatatrans}}
The relation between the initial data
\begin{equation} \label{idatatransA}
(\ub^K,\del{\tb}\ub^K)=(\vb^K,\wb^K) \quad \text{in $\Sigmab_{r_0}$}
\end{equation}
for the semilinear wave equations \eqref{Mbwave} and the corresponding initial data 
\begin{equation*} 
(u^K,\del{t} u^K)=(v^K,w^K) \quad \text{in $\Sigma_{r_0}$}
\end{equation*}
for the conformal wave equations \eqref{MwaveA} is given by
\begin{equation*}
v^K(r,\theta,\phi)=\frac{1}{r}\vb^K\biggl(\frac{1}{r},\theta,\phi\biggr) \AND w^K(r,\theta,\phi)=-\frac{1}{r^2}\wb^K\biggl(\frac{1}{r},\theta,\phi\biggr)
\end{equation*}
as can be readily verified with the help of \eqref{psibdef}, \eqref{utdef} and \eqref{utK2uK}. The initial data for the conformal wave equations, in turn, determines
via \eqref{U0def}, \eqref{V0Kdef}, \eqref{V1def} and \eqref{rhodef} the following initial data for the system \eqref{MwaveH}:
\begin{equation}\label{idatatransB}
\tilde{V}(\rho,\theta,\phi)= \begin{pmatrix}
\frac{1}{\rho^{m}}\Bigl[\frac{1}{\rho^{m}}\del{r}\vb^K\Bigl(\frac{1}{\rho^m},\theta,\phi\Big)+\vb^K\Bigl(\frac{1}{\rho^m},\theta,\phi\Big)
-\frac{1}{\rho^{m}}\wb^K\Bigl(\frac{1}{\rho^m},\theta,\phi\Big)\Bigr] \\
-\frac{1}{\rho^{m}}\Bigl[\frac{1}{\rho^{m}}\del{r}\vb^K\Bigl(\frac{1}{\rho^m},\theta,\phi\Big)+\vb^K\Bigl(\frac{1}{\rho^m},\theta,\phi\Big)
+\frac{1}{\rho^{m}}\wb^K\Bigl(\frac{1}{\rho^m},\theta,\phi\Big)\Bigr]\\
\frac{\sqrt{2}}{\rho^m}\del{\theta}\vb^K\Bigl(\frac{1}{\rho^m},\theta,\phi\Big)\\
 \frac{\sqrt{2}}{\rho^m}\del{\phi}\vb^K\Bigl(\frac{1}{\rho^m},\theta,\phi\Big)\\
 \frac{1}{\rho^m}\vb^K\Bigl(\frac{1}{\rho^m},\theta,\phi\Big)
\end{pmatrix},
\end{equation} 
which, of course, satisfies the constraint \eqref{idataconstraint}. By the above discussion, we can extend this data in any matter we like to
$\Sc$ to obtain initial data for the extended system \eqref{MwaveJ.1}, and thus, we can choose any initial $\mathring{V}$
for \eqref{MwaveJ.1}  on $\Sc$ satisfying
\begin{equation} \label{idatatransC}
\mathring{V}|_{\Sigma_{r_0}}=\tilde{V}
\end{equation}
in order to obtain solutions to \eqref{Mbwave} on $\Mb_{r_0}$ that are uniquely determined by the initial data \eqref{idatatransA}.

\subsection{The differentiated system}
While the extended system \eqref{MwaveI} is in Fuchsian form, it is not yet in a form that is required in order to apply the
Fuchsian GIVP existence
theory developed in \cite{BOOS:2020}. To obtain a system that is in the required form,  we need to
modify \eqref{MwaveI} and complement it with a differentiated version. The differentiated version is obtained by applying the 
Levi-Civita connection $\Dc_j$ of the Riemannian metric\footnote{See Appendix \ref{indexing} for our indexing conventions.}
\begin{equation} \label{qdef}
q= q_{ij}dy^i \otimes dy^j:= d\rho\otimes d\rho + \gsl, \quad y=(y^i):=(\rho,\theta,\phi),
\end{equation}
on $\Sc$.
Noting that
\begin{equation} \label{Dcdef}
\Dc_i = \delta_i^1 \del{\rho}+\delta_i^\Lambda \nablasl{\Lambda},
\end{equation}
where we recall that $\nablasl{\Lambda}$ is the Levi-Civita connection of the metric $\gsl_{\Lambda\Sigma}$ on $\mathbb{S}^{2}$,
we see after a short calculation that applying $\Dc_j$ to \eqref{MwaveI} and multiplying
the result by $t^{\kappa}$, where $\kappa \geq 0$ is a constant to be fixed below,
yields  
\begin{align}
B^0\del{t}W^K_j + \frac{1}{t} \frac{\chi \rho}{m}B^1\del{\rho}W^K_j + \frac{1}{t^{\frac{1}{2}}}
B^\Sigma \nablasl{\Sigma}W^K_j
&= \frac{1}{t}\bigl(\Bc\Pbb+\kappa B^0\bigr)W^K_j + \frac{1}{t}\Qc_j^K+
\Hc_j^K\label{MwaveK}
\end{align}
where
\begin{equation} \label{WKdef}
W^K_j =  (W^{K}_{j\Ic}) := \bigl(t^{\kappa} \Dc_j V^K_\Ic\bigr),
\end{equation}
\begin{equation} \label{Qcdef}
 \Qc_j^K = - t^{\kappa}2\chi(\rho)\rho^m\bb^{K}_{IJ} \Dc_j( V^I_{0} V_0^J)\ev_0
\end{equation}
and
\begin{align} 
\Hc^K_j = \Cc W^K_j& + t^{\kappa-\frac{1}{2}}B^\Sigma[\nablasl{\Sigma},\Dc_j]V^K 
- \frac{1}{t}\del{\rho}\biggl(\frac{\chi\rho}{m}B^1\biggr)
\delta_j^1 W^K_1 \notag \\
&+ t^{\kappa}\Dc_j \Gc
-  t^{\kappa-1}2\Dc_j(\bb^{K}_{IJ}\chi\rho^m) V^I_0 V_0^J \ev_0.\label{HcKdef}
\end{align}
It is worthwhile pointing out that the term $[\nablasl{\Sigma},\Dc_j]V^K$ does not involve any differentiation since the commutator
can be expressed completely in terms of the curvature of the metric $\gsl_{\Lambda\Sigma}$.

\subsection{The asymptotic equation}
The next step in the derivation of a suitable Fuchsian equation involves modifying the $V^K_0$ component of the 
extended system \eqref{MwaveI} given by 
\begin{equation} \label{MwaveL}
(2-t)\del{t}V_0^K =-\frac{2}{t}\chi \rho^m \bb^K_{IJ} V_0^I V_0^J + V_0^K- \frac{1}{t^\kappa}\frac{\chi \rho}{m}W^K_{10}
+\frac{1}{t^{\frac{1}{2}+\kappa}\ptt}\gsl^{\Sigma\Lambda}W^K_{\Lambda\Sigma} + \Gc_{0}^K,
\end{equation}
where  $\Gc^K = (\Gc_{\Ic}^K)$,
in order to remove the singular term
$\frac{1}{t}Q^K$. We remove this singular term using the flow\footnote{Note that the flow depends on $y=(y^i)=(\rho,\theta,\phi)\in \Sc$ through
the coefficients $\chi \rho^m \bb^K_{IJ}$, which are smooth functions on $\Sc$.}
$\Fsc(t,t_0,y, \xi)=(\Fsc^K(t,t_0,y,\xi))$ of the asymptotic equation \eqref{asympeqn}, i.e.
\begin{align}
(2-t)\del{t}\Fsc(t,t_0,y,\xi) &=  \frac{1}{t}Q\bigl(\Fsc(t,t_0,y,\xi)\bigr), \label{asympIVP.1} \\
\Fsc(t,t_0,y,\xi)&= \xi. \label{asympIVP.2}
\end{align}
Before proceeding, we note that, for fixed $(t,t_0,y)$, the flow $\Fsc(t,t_0,y, \xi)$ maps $\Rbb^{N}$ to itself, and consequently, the derivative $D_\xi F(t,t_0,y,\xi)$ 
defines a linear map from $\Rbb^{N}$ to itself, or equivalently, a $N\!\times\!N$-matrix.

Using the asymptotic flow, we define a new set of variables $Y(t,y)=(Y^K(t,y))$ via
\begin{equation}\label{Ydef}
V_0(t,y) = \Fsc(t,1,y,Y(t,y))
\end{equation}
where
\begin{equation}\label{V0def}
V_0= (V^K_0).
\end{equation}
A short calculation involving \eqref{MwaveL} and \eqref{asympIVP.1} then shows that $Y$ satisfies
\begin{equation}  \label{MwaveM}
(2-t)\del{t}Y= \Lsc\Gsc
\end{equation}
where
\begin{equation} \label{Lscdef}
 \Lsc =(D_{\xi}\Fsc(t,1,y,Y))^{-1} \AND
\Gsc =  \Biggl(V_0^K- \frac{1}{t^\kappa}\frac{\chi \rho}{m}W^K_{10}
+\frac{1}{t^{\frac{1}{2}+\kappa}\ptt}\gsl^{\Sigma\Lambda}W^K_{\Lambda\Sigma} + \Gc_{0}^K\Biggr).
\end{equation}

\subsubsection{Asymptotic flow assumptions\label{Asymptoticflowassump}}

We now assume that the flow $\Fsc(t,t_0,y,\xi)=(\Fsc^K(t,t_0,y,\xi))$ satisfies the following: for any $\Ntt\in \Zbb_{\geq 0}$, there exist constants $R_0>0$, $\ep \in [0,1/10]$ and $C_{k\ell }>0$, where
$k,\ell\in \Zbb_{\geq 0}$ and $0 \leq k+\ell\leq \Ntt$, and a function $\omega(R)$ satisfying $\lim_{R\searrow 0} \omega(R)=0$
such that
\begin{gather} 
\bigl|\Fsc(t,1,y,\xi)\big|\leq \omega(R)  \label{flowassump.1}
\intertext{and}
\bigl| D_\xi^k \Dc^\ell \Fsc(t,1,y,\xi)\big| 
+\bigl|D_\xi^k \Dc^\ell  \bigl(D_\xi\Fsc(t,1,y,\xi)\bigr)^{-1}\big| 
\leq \frac{1}{t^{\ep}}C_{k\ell}  \label{flowassump.2}
\end{gather}
for all $(t,y,\xi) \in (0,1]\times \Sc \times B_R(\Rbb^{N})$ and $R\in (0,R_0]$. A direct consequence of this assumption is that for any $\sigma>0$ the
maps
$\Ftt$ and $\check{\Ftt}$ 
defined by
\begin{equation} \label{flowassump1}
\Ftt(t,y,\xi)=t^{\ep+\sigma} \Fsc(t,1,y,\xi) \AND  \check{\Ftt}(t,y,\xi)=t^{\ep+\sigma} \bigl(D_\xi\Fsc(t,1,y,\xi)\bigr)^{-1},
\end{equation}
respectively, satisfy
$\Ftt \in C^0\bigl([0,1],C^{\Ntt}(\Sc\times B_R(\Rbb^N),\Rbb)\bigr)$ and 
$\check{\Ftt} \in C^0\bigl([0,1],C^{\Ntt}(\Sc\times B_R(\Rbb^N),\Mbb{N})\bigr)$.
Furthermore, since $\xi=0$ obviously solves the asymptotic equation \eqref{asympeqn},
the flow obviously satisfies $\Fsc(t,t_0,y,0)=0$, which in turn, implies that
\begin{equation} \label{flowassump2}
\Ftt(t,y,0)=0 
\end{equation}
for all $(t,y)\in [0,1]\times \Sc$. We further note if the asymptotic flow assumptions are satisfied for some $\epsilon\in [0,1/10)$, then they will continue to be
satisfied for all $\tilde{\epsilon}\in(\epsilon,1/10]$. Consequently, by increasing $\epsilon$ slightly, we are free to 
replace $\sigma+\epsilon$ in \eqref{flowassump1} by $\epsilon$.

\begin{prop} \label{asympprop}
Suppose the bounded weak null condition holds (see Definition \ref{bwnc}). Then there exists a $R_0\in (0,\Rc_0)$
such that the flow $\Fsc(t,t_0,y,\xi)$ of the asymptotic equation \eqref{asympeqn} satisfies the flow assumptions
\eqref{flowassump.1}-\eqref{flowassump.2} for this choice of $R_0$ and any choice of $\ep \in (0,1/10]$.
\end{prop}
\begin{proof}
We begin the proof by first establishing the following lemma that gives an effective bound on solutions of the asymptotic equation.
\begin{lem} \label{asymplem}
For any $R\in (0,\Rc_0]$, the solutions $\xi$ of the asymptotic IVP \eqref{asympprop1.1}-\eqref{asympprop1.2} exist
for $t\in (0,1]$ and satisfies
\begin{equation}
\sup_{0<t\leq 1}|\xi(t)| \leq \frac{C}{\Rc_0} R \label{asymplem0}
\end{equation}
for any choice of initial data that is bounded by $|\mathring{\xi}| \leq R$.
\end{lem}
\begin{proof}
Since  $Q(\xi)$, see \eqref{Qmap}, is independent of $t$, we can make the asymptotic equation
autonomous through the introduction of the new time variable $\tau = -\frac{1}{2}\ln(2-t) + \frac{1}{2}\ln(t)$,
which maps the time interval $0<t\leq 1$ to $-\infty <\tau \leq 0$.
In terms this new time variable $\tau$, the asymptotic IVP \eqref{asympprop1.1}-\eqref{asympprop1.2} becomes
\begin{align}
\del{\tau}\xi &= Q(\xi), \label{asymplem1.1} \\
\xi|_{\tau=0} &= \mathring{\xi}. \label{asymplem1.2}
\end{align}
By the bounded weak null condition, this IVP admits solutions that are defined for $\tau \in (-\infty,0]$ and satisfy
\begin{equation}\label{asymplem2}
\sup_{-\infty<\tau\leq 0}|\xi(\tau)| \leq C
\end{equation}
provided that  $|\mathring{\xi}| < \Rc_0$. Next, we assume that the initial value $\mathring{\xi}$ satisfies 
$|\mathring{\xi}| < R$ 
for some $R\in (0,\Rc_0]$, and we set
$\tilde{\xi}(\tau) = \frac{1}{\rf}\xi\bigl(\frac{\tau}{\rf}\bigr)$ where
$\rf=\frac{R}{\Rc_0}\in (0,1]$.
Then a  quick calculation shows that $\xit$ satisfies asymptotic equation \eqref{asymplem1.1} where that
initial value is bounded by 
$|\tilde{\xi}|_{\tau=0}| = \bigl|\frac{1}{\rf}\mathring{\xi}\bigr| < \frac{\Rc_0}{R} R = \Rc_0$, 
and consequently, we deduce from \eqref{asymplem2} that $\sup_{-\infty<\tau\leq 0}|\tilde{\xi}(\tau)| \leq C$.
But this implies that $\sup_{-\infty<\tau\leq 0}|\xi(\tau)| \leq \frac{C}{\Rc_0}R$, and the proof of the lemma is complete.
\end{proof}
Implicitly, the solution $\xi=(\xi^K)$ depends on $y\in \Sc$ and the initial data $\mathring{\xi}$. Fixing $\ep >0$ and differentiating the asymptotic
equation \eqref{asympeqn} with respect to $y=(y^i)$ shows that 
\begin{equation}\label{asympprop3}
\eta^K_i= t^{\ep}\Dc_i\xi^K
\end{equation}
satisfies the differential equation
\begin{equation} \label{asympprop3a}
(2-t)\del{t} \eta^K_i =  \frac{1}{t}\bigl( (2-t)\ep \delta^K_J  - 2\chi \rho^m\bigl(\bb^K_{JI}
+ \bb^K_{IJ}\bigr)\xi^I \bigr) \eta^J_i - \frac{1}{t^{1-\ep}}\Dc_i(2\chi \rho^m \bb^K_{IJ}) \xi^I \xi^J.   
\end{equation}
Contracting this equation with $\delta_{LK}\delta^{ki}\eta_k^L$ gives
\begin{equation*}
(2-t)\delta_{LK}\delta^{ki}\eta_k^L \del{t} \eta^K_i = \frac{1}{t}\delta_{LK}\delta^{ki}\eta_k^L
\bigl( (2-t)\ep \delta^K_J  - 2\chi \rho^m\bigl(\bb^K_{JI}
+ \bb^K_{IJ}\bigr)\xi^I \bigr) \eta^J_i -\frac{1}{t^{1-\ep}}\delta_{LK}\delta^{ki}\eta_k^L\Dc_i(2\chi \rho^m \bb^K_{IJ}) \xi^I \xi^J.
\end{equation*}
Letting $|\eta| = \sqrt{\delta_{KL}\delta^{ij}\eta^K_i \eta^L_j}$,
denote the Euclidean norm of $\eta=(\eta_i^K)$, we can write the above equation as
\begin{equation} \label{asympprop4}
\frac{(2-t)}{2}\del{t}|\eta|^2 = \frac{1}{t}\bigl((2-t)\ep |\eta|^2 -2\chi \rho^m\bigl(\bb^K_{JI}
+ \bb^K_{IJ}\bigr) \delta_{LK}\xi^I  \delta^{ki}\eta_k^L\eta^J_i \bigr) -\frac{1}{t^{1-\ep}}\delta_{LK}\delta^{ki}\eta_k^L\Dc_i(2\chi \rho^m \bb^K_{IJ}) \xi^I \xi^J.
\end{equation}
But $\chi \rho^m$ and $\bb^K_{IJ}$ are smooth on $\Sc$, and consequently, these functions and their derivatives are bounded on
$\Sc$. From this fact and the bound on $\xi$ from Lemma \ref{asymplem}, we deduce from  \eqref{asympprop4}
and the Cauchy Schwartz inequality that for any $\sigma \in (0,\ep)$ there exists constants $R_0\in (0,\Rc_0]$ and $C>0$ such that the energy inequality
\begin{equation*}
\frac{(2-t)}{2}\del{t}|\eta|^2 \geq \frac{(2-t)}{t}(\ep-\sigma) |\eta|^2 -\frac{C}{t^{1-\ep}}|\eta|
\end{equation*}
holds for any given $R\in (0,R_0]$ and for all $t\in (0,1]$. But from this inequality, we see that 
\begin{equation*}
\del{t}|\eta| \geq \frac{1}{t}(\ep-\sigma) |\eta| -\frac{C}{t^{1-\ep}}.
\end{equation*}
An application of Gr\"{o}nwall's inequality\footnote{Here, we are using the following form of Gr\" {o}nwall's inequality: if $x(t)$ satisfies
$x'(t) \geq a(t)x(t) -h(t)$, $0<t\leq T_0$, then $x(t)\leq x(T_0)e^{-A(t)}+ \int^{T_0}_t e^{-A(t)+A(\tau)}h(\tau) \, d\tau$ where
$A(t)= \int^{T_0}_t a(\tau)\, d\tau$. In particular, we observe from this that if, $x(T_0)\geq 0$ and  $a(t)=\frac{\lambda}{t}-b(t)$, where  $\lambda \in \Rbb$ and $\bigl|\int^{T_0}_t b(\tau) \, d\tau \bigr|\leq r$, then
\begin{equation*}
x(t) \leq e^{r}x(T_0)\left(\frac{t}{T_0}\right)^\lambda +   e^{2r}t^\lambda \int^{T_0}_t \frac{|h(\tau)|}{\tau^\lambda}\, d\tau
\end{equation*}
for $0\leq t < T_0$.
} then yields
\begin{equation} \label{asympprop4a}
|\eta(t)| \leq |\eta(1)| t^{\ep-\sigma}+   t^{\ep-\sigma}\int^1_t \frac{C}{t^{1-\sigma}}\, d\tau = 
t^{\ep-\sigma}|\eta(1)| +\frac{1}{\sigma}t^{\ep-\sigma}\bigl(1-t^\sigma\bigr).
\end{equation}
From the definition \eqref{asympprop3} and the fact that $\xi(t)=\Fsc(t,1,y,\mathring{\xi})$, we conclude from the above inequality and \eqref{asymplem0} that there
exist constants $C_0, C_{01}>0$ such that the flow $\Fsc$ satisfies the bounds
\begin{equation*}  
|\Fsc(t,1,y,\mathring{\xi})| \leq C_0 R %\label{asympprop5.1}
\AND
|\Dc\Fsc(t,1,y,\mathring{\xi})| \leq \frac{1}{t^{\sigma}}C_{01} %\label{asympprop5.2}
\end{equation*}
for all $(t,y,\mathring{\xi})\in (0,1]\times \Sc \times B_{R}(\Rbb^N)$, $R\in (0,R_0]$.

Next, differentiating the asymptotic
equation \eqref{asympeqn} with respect to the initial data $\mathring{\xi}$ shows
that the derivative 
\begin{equation*}
 D_{\mathring{\xi}}\xi= \biggl( \frac{\del{} \xi^K}{\del{} \mathring{\xi}^L}\biggr)
\end{equation*}
satisfies the equation 
\begin{equation} \label{asympprop5}
(2-t)\del{t} D_{\mathring{\xi}}\xi =  \frac{1}{t}L D_{\mathring{\xi}}\xi
\end{equation}
where 
\begin{equation*}
L=(L^K_J):=-2\chi \rho^m\bigl(\bb^K_{JI}
+ \bb^K_{IJ}\bigr)\xi^I \bigr).
\end{equation*}
Furthermore, multiplying \eqref{asympprop5} on the right by  $(D_{\mathring{\xi}}\xi)^{-1}$ yields the equation
\begin{equation} \label{asympprop6}
(2-t)\del{t} ( (D_{\mathring{\xi}}\xi)^{-1})^{\tr} =  -\frac{1}{t} 
L^{\tr} ((D_{\mathring{\xi}}\xi)^{-1})^{\tr}
\end{equation}
for the transpose of $(D_{\mathring{\xi}}\xi)^{-1}$.
Multiplying \eqref{asympprop5} and \eqref{asympprop6} by
$t^\ep$, we find that
\begin{align*}
(2-t)\del{t}( t^\ep D_{\mathring{\xi}}\xi) &=  \frac{1}{t}
\bigl((2-t)\ep + L\bigr) t^\ep D_{\mathring{\xi}}\xi  %\label{asympprop7}
\intertext{and}
(2-t)\del{t} (t^\ep (D_{\mathring{\xi}}\xi)^{-1})^{\tr} &=  \frac{1}{t}
\bigl((2-t)\ep-L^{\tr}\bigr)(t^\ep (D_{\mathring{\xi}}\xi)^{-1})^{\tr}.
 %\label{asympprop8}
\end{align*}
Both of the these equations
are of the same general form as \eqref{asympprop3a}, and
the same arguments used to derive from \eqref{asympprop3a} the
bounds \eqref{asympprop4a} for $\eta = t^\ep \Dc \xi$
can be used to obtain similar estimates for
$t^\ep D_{\mathring{\xi}}\xi$ and $(t^\ep(D_{\mathring{\xi}}\xi)^{-1})^{\tr}$. Consequently, shrinking $R_0$ if necessary and arguing as above, we deduce the existence of a constant $C_{10}>0$
such that the estimate 
\begin{equation*}
\bigl|D_{\mathring{\xi}}\xi\bigr|+ \bigl|(D_{\mathring{\xi}}\xi)^{-1}\bigr|\leq \frac{1}{t^{\sigma}}C_{10}
\end{equation*}
holds for $0<t\leq 1$. From this estimate, we see
immediately that
\begin{equation*} %\label{asympprop6}
\bigl|D_{\mathring{\xi}}\Fsc(t,1,y,\mathring{\xi})\bigr|
+\bigl|\bigl(D_{\mathring{\xi}}\Fsc(t,1,y,\mathring{\xi})\bigr)^{-1}\bigr|\leq \frac{1}{t^{\sigma}}C_{10},
\end{equation*}
for all $(t,y,\mathring{\xi})\in (0,1]\times \Sc \times B_{R}(\Rbb^N)$ and $R\in (0,R_0]$. 

Finally, by shrinking $R_0$ again if necessary, similar arguments as above can be used to derive, for any fixed $\Ntt \in \Zbb_{\geq 1}$, the bounds
\begin{equation*}
\bigl|D^k_{\mathring{\xi}}\Dc^\ell\xi\bigr|
+\bigl|D^k_{\mathring{\xi}}\Dc^\ell
\bigl(D_{\mathring{\xi}}\xi)^{-1}\bigr|\leq \frac{1}{t^\sigma}C_{kl} 
\end{equation*}
on the higher derivatives for $1\leq k + \ell \leq \Ntt$. It
is then clear from this inequality that the flow bounds
\begin{equation*}
|D_{\mathring{\xi}}^k \Dc^\ell \Fsc(t,1,y,\mathring{\xi})| \leq \frac{1}{t^{\sigma}}C_{\ell k},
\end{equation*}
hold for all $(t,y,\mathring{\xi})\in (0,1]\times \Sc \times B_{R}(\Rbb^N)$, $2\leq k + \ell \leq \Ntt$, and $R\in (0,R_0]$. This completes the proof of the proposition.
\end{proof}

\subsection{The complete Fuchsian system\label{complete:sec}}
We complete the derivation of the Fuchsian equation by 
complimenting \eqref{MwaveK}
and \eqref{MwaveM} with a third system obtained from applying the projection operator $\Pbb$ to \eqref{MwaveI}, which
leads to an equation for the variables
\begin{equation}\label{XKdef}
X^K = \frac{1}{t^\nu}\Pbb V^K,
\end{equation}
where $\nu\geq 0$ is a constant to be fixed below. Now, a straightforward calculation using \eqref{Pbbdef}, \eqref{Pbbprops}-\eqref{Pbbcom}, 
%\eqref{Pbbperpdef}, 
\eqref{FcKdef},  \eqref{Gcdef} and
\eqref{WKdef} shows that after multiplying \eqref{MwaveI} by $t^{-\nu}\Pbb$ that  $X^K$ satisfies
\begin{align}
B^0\del{t}X^K + \frac{1}{t} \frac{\chi \rho}{m}B^1\del{\rho}X^K 
&= \frac{1}{t}(\Bc - \nu B^0) X^K +\Kc^K
\label{MwaveN}
\end{align}
where
\begin{equation}\label{KcKdef}
\Kc^K=- \frac{1}{t^{\frac{1}{2}+\kappa+\nu}}\Pbb B^\Sigma W_\Sigma^K +\Pbb\Cc \biggl(\frac{1}{t^{\nu}}\Pbb^\perp V^K + X^K\biggr) + \frac{1}{t^\nu}\Pbb \Gc_0^K + \frac{1}{t^{\frac{1}{2}+\nu}}\Pbb \Gc^K_1 
\end{equation}
and
\begin{equation} \label{Pbbperpdef}
\Pbb^\perp = \id - \Pbb
\end{equation}
is the complementary projection oprator. We now complete our derivation of the Fuchsian equation, which will be crucial for our existence proof, by collecting  \eqref{MwaveK}, \eqref{MwaveM} and \eqref{MwaveN} into the following single system:
\begin{align}
A^0\del{t}Z + \frac{1}{t}\frac{\chi \rho}{m}A^1 \del{\rho}Z + \frac{1}{t^{\frac{1}{2}}}
A^\Sigma \nablasl{\Sigma}Z
&= \frac{1}{t}\Ac \Pi Z +  \frac{1}{t}\Qc + \Jc \label{MwaveO}
\end{align}
where
\begin{align}
Z &=  \begin{pmatrix} W^K_j & X^K & Y \end{pmatrix}^{\tr} , \label{Zdef}\\
A^0 &= \begin{pmatrix} B^0 & 0 & 0 \\ 0 & B^0 & 0 \\ 0 & 0 & (2-t)\id \end{pmatrix}, \label{A0def} \\
A^1 & = \begin{pmatrix}  B^1\delta^j_k \delta_K^L & 0 & 0 \\ 0 & B^1\delta_K^L & 0 \\ 0 & 0 &0 \end{pmatrix}, \label{A1def} \\
A^\Sigma & = \begin{pmatrix}  B^\Sigma  & 0 & 0 \\ 0 & B^\Sigma  & 0 \\ 0 & 0 & 0  \end{pmatrix}, \label{ASigmadef} \\
\Ac & = \begin{pmatrix} \Bc\Pbb + \kappa B^0 & 0 & 0 \\ 0 & \Bc-\nu B^0 & 0 \\ 0 & 0 & 2\id  \end{pmatrix}, \label{Acdef} \\
\Pi & = \begin{pmatrix}  \id & 0 & 0 \\ 0 & \id & 0 \\ 0 & 0 & 0  \end{pmatrix}, \label{Pidef}\\
\Qc &= \begin{pmatrix} \Qc^K_j & 0 & 0 \end{pmatrix}^{\tr} \label{Qcvecdef}
\intertext{and}
\Jc &= \begin{pmatrix}\Hc^K_j &  \Kc^K & \Lsc\Gsc \end{pmatrix}^{\tr}.\label{Jcdef}
\end{align}

\subsection{Coefficient properties\label{coefprops}}
We now turn to verifying that the system \eqref{MwaveO} satisfies all the assumptions needed to apply the Fuchsian GIVP existence
theory from \cite{BOOS:2020}.

\subsubsection{The projection operator $\Pi$ and its commutation properties:}
By construction, the field $Z$, defined by \eqref{Zdef}, is a time-dependent section of the vector bundle
\begin{equation*}
\Wbb = \bigcup_{y\in \Sc} \Wbb_{y}
\end{equation*}
over $\Sc$ with fibers
$\Wbb_y =\Bigl(T^*_{y}\Sc \times T^*_{y}\Sc \times \bigl(\text{T}^*_{y}\Sc \otimes T^*_{\text{pr}(y)}\mathbb{S}^2\bigr)\times T^*_{y}\Sc\Bigr)^N\times \Vbb_y^N\times \Rbb^N$ 
where, as above, $\text{pr} :  \Sc  \longrightarrow \mathbb{S}^2$ is the canonical projection and $\Vbb_y=\Rbb\times \Rbb\times \text{T}^*_{\text{pr}(y)}\mathbb{S}^2\times\Rbb$.
Letting $\grave{Z} = ( \grave{W}{}^K_j, \grave{X}{}^K , \grave{Y})$ and $Z$ be as defined above by \eqref{Zdef}, we introduce an inner-product on $\Wbb$ via
\begin{equation} \label{hcdef}
\hc(Z,\grave{Z}) = \delta_{KL}q^{ij}h(W^K_i, \grave{W}{}^L_j)
+ \delta_{KL}h(X^K,\grave{X}{}^L)+\delta_{KL}Y^K \grave{Y}^L,
\end{equation} 
where $h(\cdot,\cdot)$ is the inner-product defined previously by \eqref{hdef}. It is then not difficult to verify that this inner-product is compatible, i.e. $\Dc_j\bigl(\hc(Z,\grave{Z})\bigr)= \hc(\Dc_j Z,\grave{Z})+\hc(Z, \Dc_j \grave{Z})$, with the connection $\Dc_j$ defined above by \eqref{Dcdef}. We further observe from \eqref{Pidef} that
$\Pi$ defines a projection operator, i.e.
\begin{equation} \label{Piproj}
\Pi^2 = \Pi ,
\end{equation}
that is symmetric with respect to the inner-product $\eqref{hcdef}$. It also follows directly from the definitions \eqref{A0def}, \eqref{A1def} and \eqref{Acdef} that
\begin{equation} \label{Picom}
[A^0,\Pi]=[\Ac,\Pi]=0,
\end{equation}
\begin{equation} \label{PipAproj}
\Pi A^1 =  A^1 \Pi =  A^1, \quad \Pi A^\Sigma \eta_\Sigma =A^\Sigma \eta_\Sigma\Pi  = A^\Sigma\eta_\Sigma ,
\end{equation}
and
\begin{equation} \label{PiperpAproj}
\Pi^\perp A^1 =  A^1 \Pi^\perp = \Pi^\perp A^\Sigma \eta_\Sigma =A^\Sigma \eta_\Sigma\Pi^\perp  =0,
\end{equation}
where
\begin{equation*}
\Pi^\perp = \id -\Pi
\end{equation*}
is the complementary projection operator.

\subsubsection{The operators $A^0$, $A^1$, $A^\Sigma n_\Sigma$ and $\Ac$:}
Next, we see from \eqref{B0lowbnd}, \eqref{A0def} and \eqref{hcdef} that $A^0$ satisfies
\begin{align*}
\hc(Z,A^0 Z) &= \delta_{KL}q^{ij}h(W^K_i,B^0 W^L_{j})
+ \delta_{KL}h(X^K,B^0 X^L)+(2-t)\delta_{KL}Y^K Y^L \\
 &\geq \delta_{KL}q^{ij}h(W^K_i,W^L_j)
+ \delta_{KL}h(X^K,X^L)+(2-t)\delta_{KL}Y^K Y^L,
\end{align*}
and hence, that
\begin{equation} \label{gammabnd}
\hc(Z,Z)\leq \hc(Z,A^0 Z).
\end{equation}
Similar calculations using \eqref{Pbbprops}-\eqref{Pbbcom}, \eqref{B0Bcbnd}, \eqref{A0def}, \eqref{Acdef} and \eqref{hcdef}
show that 
\begin{equation} \label{kappabnd}
\kappa \hc(Z,A^0 Z) \leq \hc(Z,\Ac Z)
\end{equation} 
provided that  $\nu,\kappa\geq 0$ and $\kappa+\nu\leq 1/2$.
It is also clear from \eqref{A0def}-\eqref{ASigmadef} that $A^0$, $A^1$ and $A^\Sigma \eta_\Sigma$ are symmetric with respect to the inner-product \eqref{hcdef}.
%, that is,
%\begin{equation} \label{Asym}
%(A^0)^{\tr}=A^0, \quad (A^1)^{\tr}=A^1 \AND (A^\Sigma \eta_\Sigma)^{\tr}= A^\Sigma \eta_\Sigma.
%\end{equation}
Finally,  we observe that the inequality 
\begin{equation*}
\biggl|\del{\rho}\biggl(\frac{\chi\rho}{m}B^1\biggr)\biggr| \leq \max_{0\leq t \leq 1}|B^1(t)|
\norm{\del{\rho}(\chi\rho)}_{L^\infty(\Tbb)}\frac{1}{m}
\end{equation*}
follows easily from \eqref{B1def} and \eqref{qdef}.
With the help of this inequality, we deduce from \eqref{A1def} that, for any given $\sigma >0$, there exists an integer $m=m(\sigma)\geq 1$ such
that 
\begin{equation} \label{A1bnd}
\biggl|\del{\rho}\biggl(\frac{\chi\rho}{m}B^1\biggr)\biggr| + \biggl|\del{\rho}\biggl(\frac{\chi\rho}{m} A^1\biggr)\biggr| < \sigma \quad \text{in $(0,1)\times \Sc$}.
\end{equation}

\subsubsection{The source term $\Jc$:}
Using \eqref{Vdef},  \eqref{Pbbdef}, \eqref{Ydef}-\eqref{V0def}, \eqref{flowassump1} and \eqref{XKdef}, we can decompose  
$V^K$ as
\begin{equation} \label{VKdecomp1}
V^K(t,y) = \Pbb V^K(t,y) + \Pbb^\perp V^K(t,y),
\end{equation}
where
\begin{gather}
\Pbb V^K(t,y) = t^\nu X^K(t,y) \label{VKdecomp2}
\intertext{and}
 \Pbb^\perp V^K(t,y) =  \frac{1}{t^{\ep}}( t^\ep V_0^K(t,y))\ev_0 =  
  \frac{1}{t^{\ep}}\Ftt^K\bigl(t,y,Y(t,y)\bigr)  \bigr)\ev_0,\label{VKdecomp3}
\end{gather}
while we recall from \eqref{WKdef} that the derivative $\Dc_j V^K$ is determined by
\begin{equation} \label{VKdecomp4}
\Dc_j V^K(t,y) = W_j^K(t,y).
\end{equation}  
We further observe from \eqref{Lscdef} and \eqref{flowassump1} that the map $\Lsc$ can be expressed as
\begin{equation} \label{Lscmap}
\Lsc = \frac{1}{t^\ep}\check{\Ftt}\bigl(t,y,Y(t,y)\bigr).
\end{equation}

Now, setting
\begin{equation*}% \label{Xdef}
X=(X^K),
\end{equation*}
we can use \eqref{VKdecomp1}-\eqref{VKdecomp3} along with \eqref{GcK0def}-\eqref{GcK1def} to write the source
term \eqref{KcKdef} as
\begin{align*}
\Kc^K = - \frac{1}{t^{\frac{1}{2}+\kappa+\nu}}&\Pbb B^\Sigma(t,y) W_\Sigma^K(t,y) +\frac{1}{t^{\nu+\ep}}\Ftt^K\bigl(t,y,Y(t,y) \bigr)
\Pbb\Cc(t)\ev_0 + \Pbb\Cc(t) X^K(t,y) \\
+& \frac{1}{t^{\nu+2\ep}}
 \Pbb \Gc^K_0\Bigl(t^{\frac{1}{2}},t,\chi(\rho)\rho^m,\Ftt\bigl(t,y,Y(t,y)\bigr)\ev_0 ,\Ftt\bigl(t,y,Y(t,y)\bigr)\ev_0 \Bigr) \\
 +&  \sum_{a=0}^1\biggl\{ \frac{1}{t^{\frac{a}{2}+\ep}}\biggr[
 \Pbb \Gc^K_a\Bigl(t^{\frac{1}{2}},t,\chi(\rho)\rho^m,\Ftt\bigl(t,y,Y(t,y)\bigr)\ev_0 ,X(t,y)\Bigr) \\
 +& \Pbb \Gc^K_a\Bigl(t^{\frac{1}{2}},t,\chi(\rho)\rho^m,X(t,y),\Ftt\bigl(t,y,Y(t,y)\bigr)\ev_0 \Bigr)   \biggr] +  \frac{1}{t^{\frac{a}{2}-\nu}}\Pbb \Gc^K_a\Bigl(t^{\frac{1}{2}},t,\chi(\rho)\rho^m,X(t,y),X(t,y)\Bigr)   \biggr\}.
\end{align*}
Using \eqref{VKdecomp1}-\eqref{Lscmap} to similarly express the source terms $\Hc^K_j$ and $\Lsc\Gsc$, see   \eqref{HcKdef} and \eqref{Lscdef}, in  
terms of $W^K_j$, $X^K$ and $Y^K$, it is then not difficult, with the help of \eqref{flowassump2},
\eqref{A1bnd} and the assumptions $\ep,\kappa,\nu\geq 0$, that we can expand the source 
term \eqref{Jcdef} as
\begin{align}
\Jc = \biggl(\frac{1}{t^{3\ep}}+\frac{1}{t^{\nu+2\ep}}+\frac{1}{t^{1-\kappa+2\ep}}\biggr)\Jc_0\bigl(&t,y,Z(t,y)\bigr) + \biggl(\frac{1}{t^{\frac{1}{2}+\kappa+\ep}}+\frac{1}{t^{\frac{1}{2}+2\ep-\nu}}\biggr) \Jc_1\bigl(t,y,Z(t,y)\bigr) \notag \\
&+ \frac{1}{t}\bigl(\sigma + t^{\frac{1}{2}-\kappa-\nu}+
t^{\frac{1}{2}-\ep}
+ t^{\frac{1}{2}-\kappa-\ep}+t^{2\nu-\ep}\bigr)\Jc_2\bigl(t,y,Z(t,y)\bigr)
\label{JcexpA}
\end{align}
where  $\Jc_a\in C^0([0,1], C^\Ntt(\Sc\times B_R(\Wbb),\Wbb))$, $a=0,1,2$, for any fixed $\Ntt\in \Zbb_{\geq 0}$, and these maps satisfy\footnote{Here, we are using are the order notation $\Ord(\cdot)$ from \cite[\S 2.4]{BOOS:2020} where the maps are
finitely rather than infinitely differentiable.}
\begin{gather}
\Jc_0 = \Ord(Z), \quad \Jc_1 = \Ord(\Pi Z), \quad
\Pi\Jc_2  = \Ord(\Pi Z) \AND
\Pi^\perp\Jc_2  = \Ord(\Pi Z\otimes \Pi Z). \label{Jcord} 
\end{gather}

To proceed, we choose the constants $\kappa,\nu \in \Rbb_{>0}$ to satisfy the inequalities
\begin{equation} \label{kappa-nu-fix1}
2\ep < \kappa < 1-\ep, \quad \kappa +\nu < \frac{1}{2}-\ep, \quad  \ep < 2\nu \AND \kappa \leq \frac{1}{3},
\end{equation}
which is possible since in the following we assume that the asymptotic assumptions are satisfied for some $\ep \in (0,1/10)$. For example,
if $\ep=1/11$, we could choose $\kappa=5/22$ and $\nu =1/11$.
Now, it is not difficult to verify that \eqref{kappa-nu-fix1} implies the inequalities
\begin{gather*}
3\ep \leq 1-\kappa + 2\ep, \quad \nu+2\ep \leq 1-\kappa + 2\ep, \quad  0<2\nu-\ep, \quad 
0<\frac{1}{2}-\kappa-\ep, \quad  0<\frac{1}{2}-\kappa-\nu,  \\
\frac{1}{2}+2\ep -\nu \leq 1-\frac{\kappa}{2}+\ep, \quad \frac{1}{2}+\kappa+\ep \leq 1-\frac{\kappa}{2}+\ep
\AND 0<\kappa - 2\ep \leq 1,
\end{gather*}
and that, with the help of these inequalities, we can, after suitably redefining the maps
$\Jc_a$, rewrite \eqref{JcexpA} as
 \begin{align}
\Jc = \frac{1}{t^{1-\kappa+2\ep}}\Jc_0\bigl(t,y,Z(t,y)\bigr)+\frac{1}{t^{1-\frac{\kappa}{2}+\ep}}\Jc_1\bigl(t,y,Z(t,y)\bigr) + \frac{1}{t}\bigl(\sigma + t^{\tilde{\ep}}\bigr)\Jc_2\bigl(t,y,Z(t,y)\bigr)
\label{JcexpB}
\end{align}
for some suitably small constant $\tilde{\ep}>0$.
Here, the constant $\sigma >0$ can be chosen as small as we like, and the redefined maps  $\Jc_{a}$ have the same smoothness
properties as above and satisfy \eqref{Jcord}.

\begin{rem} \label{Jcrem}
The point of the expansion \eqref{JcexpB} is that source term $\Jc$
satisfies all the assumptions from Section 3.1.(iii) of \cite{BOOS:2020} except for the following:
\begin{enumerate}
\item the differentiablity of each of the maps $\Jc_a$ is finite,
\item and $\Jc_2$ does not satisfy $\Pi\Jc_2=0$.
\end{enumerate}  
Neither of these exceptions pose any difficulties and are easily dealt with. To see why the first exception is not problematic,
we observe from arguments of \cite{BOOS:2020} that all of the results of that paper
are valid provided that the order of the differentiability of the source term is greater than $n/2+3$, where $n$ is the dimension of
the spatial manifold. Since the spatial manifold we are considering, i.e. $\Sc$, is 3-dimensional and we have established above that the maps  $\Jc_a$ are $\Ntt$-times differentiable for
any $\Ntt \in \Zbb_{\geq 0}$, it follows by taking $\Ntt > 3/2+3$ that the finite differentiability is no obstruction to applying the
results from \cite{BOOS:2020} to the Fuchsian equation \eqref{MwaveO}. In regards to the second exception, we note,  since 
$\Pi \Jc_2= \Ord(\Pi Z)$,  that the term $\frac{1}{t}(\sigma+t^{\tilde{\epsilon}})\Pi \Jc_2$ can be absorbed into
the term $\frac{1}{t}\Ac\Pi Z$ on the right hand side of the Fuchsian equation \eqref{MwaveO} via a redefinition of the operator 
$\Ac$. 
Due to the factor $\sigma+t^{\tilde{\epsilon}}$,
we can ensure, for any choice of
$\tilde{\kappa}\in (0,\kappa)$, that the redefined matrix $\Ac$ would satisfy for all
$t\in (0,t_0]$ an inequality of the form \eqref{kappabnd} with $\kappa$
replaced by $\tilde{\kappa}$ provided that $\sigma$ and $t_0$ are chosen sufficiently small. After doing this, the redefined $\Jc_2$ would
satisfy $\Pi\Jc_2=0$ as required and the source term $\Jc$ would satisfy all the assumptions needed to apply the existence theory from \cite{BOOS:2020}. 
\end{rem}

\subsubsection{The source term $\Qc$:} We now analyze the nonlinear term \eqref{Qcvecdef} (see also \eqref{Qcdef}) in more detail.
Recalling that the $\chi\rho^m \bb^K_{IJ}$ are smooth functions on $\Sc$, we can, with the help of
the product estimate \cite[Ch.~13, Prop.~3.7.]{TaylorIII:1996} and
H\"{o}lder's inequality, estimate $\Qc$  for any $s\in \Zbb_{\geq 0}$ by
%the smoothness of the functions $\chi(\rho)\rho^m \bb^K_{IJ}$ on $\Sc$ the estimate
\begin{align}
\norm{\Qc}_{H^s(\Sc)} &\lesssim t^{\kappa}\bigl( \norm{\Dc(V_0 V_0)}_{L^\infty(\Sc)}+ 
\norm{\Dc(V_0 V_0)}_{H^s(\Sc)}\bigr)
\notag \\
&\lesssim   t^{\kappa}\bigl(  \norm{V_0}_{L^\infty(\Sc)} \norm{\Dc V_0}_{L^\infty(\Sc)} +\norm{V_0}_{L^{\infty}(\Sc)} 
\norm{\Dc V_0}_{H^{s}(\Sc)} + \norm{\Dc V_0}_{L^{\infty}(\Sc)}\norm{V_0}_{L^{s}(\Sc)}\bigr)  \notag \\
& \lesssim   \norm{V_0}_{L^\infty(\Sc)} \norm{W}_{L^\infty(\Sc)} +\norm{V_0}_{L^{\infty}(\Sc)} 
\norm{W}_{H^{s}(\Sc)} + \norm{W}_{L^{\infty}(\Sc)}\norm{V_0}_{L^{s}(\Sc)}. \label{Qcbnd1}
%&\lesssim t^{\kappa} \norm{V_0}_{L^\infty(\Sc)}\norm{V_0}_{H^{k+1}(\Sc)} \notag\\
%& \lesssim t^{\kappa} \norm{V_0}_{L^\infty(\Sc)}\bigl(  \norm{V_0}_{L^2(\Sc)} + \norm{\Dc V_0}_{H^{k}(\Sc)}\bigr) \notag \\
%& \lesssim t^{\kappa} \norm{V_0}_{L^\infty(\Sc)}\norm{V_0}_{L^2(\Sc)} +  \norm{V_0}_{L^\infty}\norm{W}_{H^{k}(\Sc)},
\end{align}
Next, for $k\in\Zbb_{> 3/2}$, we let $C_{\text{Sob}}$ denote the constant that appears in the Sobolev inequality 
\cite[Ch.~13, Prop.~2.4.]{TaylorIII:1996}, that is,
\begin{equation} \label{Sobolev}
\norm{\ftt}_{L^\infty(\Sc)} \leq C_{\text{Sob}} \norm{\ftt}_{H^{k}(\Sc)}.
\end{equation} 
Then by \eqref{Ydef}, the flow bounds \eqref{flowassump.1}-\eqref{flowassump.2}, and the Sobolev and H\"{o}lder inequalities, we see 
that the inequalities
\begin{gather} 
\norm{V_0}_{L^\infty(\Sc)}+\norm{V_0}_{L^2(\Sc)}\lesssim \omega(R) \label{V0bnd}
\intertext{and}
\norm{V_0}_{L^s(\Sc)} \lesssim  \norm{V_0}_{L^2(\Sc)} + \norm{\Dc V_0}_{L^s(\Sc)} \lesssim 
\omega(R) + \norm{W}_{L^s(\Sc)},\quad s\in \Zbb_{\geq 1},  \label{V0bndAa}
\end{gather}
hold for all $t\in (0,1]$ and $\norm{Y}_{H^k}\leq R/C_{\text{Sob}}$. Using these estimates, Sobolev's inequality
and  the estimate $\norm{W}_{L^\infty(\Sc)}\lesssim \norm{W}_{L^2(\Sc)}$, which follows from H\"{o}lder's inequality, we find from setting
$s=0$ and $s=k$ in \eqref{Qcbnd1} that
\begin{gather}
\norm{\Qc}_{L^2(\Sc)}  \lesssim \omega(R) \norm{W}_{L^{2}(\Sc)} \lesssim \omega(R) \norm{\Pi Z}_{L^{2}(\Sc)} \label{Qcbnd2}
\intertext{and}
\norm{\Qc}_{H^k(\Sc)}  \lesssim \bigl( \omega(R)+ \norm{W}_{H^{k}(\Sc)}\bigr) \norm{W}_{H^{k}(\Sc)}
\lesssim \bigl( \omega(R)+  R \bigr) \norm{\Pi Z}_{H^{k}(\Sc)} \label{Qcbnd3}
\end{gather}
for all  $\norm{Z}_{H^k(\Sc)} \leq R/C_{\text{Sob}}$.
We further observe from \eqref{Pidef} and \eqref{Qcdef} that 
\begin{equation}\label{Qcbnd4}
\Pi \Qc = \Qc.
\end{equation}
%Using this and the symmetry of the projection operator $\Pi$, we see 
%that\footnote{Here $\ip{Z}{\Qc}_{L^2(\Sc)}= \int_{\Sc} \hc(Z,\Qc)\, dq$ where $dq = \det{q_{ij}}dy^1\wedge dy^2 \wedge dy^3$
%is the volume element of the spatial metric \eqref{qdef}.}
%\begin{equation} \
%\ip{\Zc}{\Qc}_{L^2(\Sc)} = \ip{\Zc}{\Pi \Qc}_{L^2(\Sc)}= \ip{Z}{\Pi^2\Qc}_{L^2(\Sc)}= \ip{\Pi Z}{\Pi\Qc}_{L^2(\Sc)}.
%\end{equation}

\begin{rem} \label{Qcrem}
The importance of the estimates \eqref{Qcbnd2}-\eqref{Qcbnd3} and the identity \eqref{Qcbnd4} is that, by an obvious 
modification of the proof of Theorem 3.8.~in \cite{BOOS:2020}, these results show that terms in 
the energy estimates for the Fuchsian equation \eqref{MwaveO} that arise due to the ``bad'' singular term $\frac{1}{t}\Qc$ can be controlled using 
the ``good'' singular
$\frac{1}{t}\Asc\Pi Z$ by choosing $\omega(R)+R$ sufficiently small, which we can do by choosing $R$ suitably small since $\lim_{R\searrow 0}\omega(R)=0$
by assumption.
%----------------- begin comment ---------------------------- 
\begin{comment}
The two key calculations are
\begin{align*}
 \bigg\langle\Dc^\ell Z \biggl| \frac{1}{t} \Dc^\ell(\Ac \Pi Z) + \frac{1}{t}\Dc^\ell{\Qc}\biggl\rangle_{L^2(\Sc)}
 &= \frac{1}{t}\biggl(
 \ip{\Dc^\ell \Pi Z}{\Ac \Dc^\ell \Pi Z)\big}_{L^2(\Sc)} +   \ip{\Dc^\ell \Pi Z}{\Dc^\ell{\Qc}}_{L^2(\Sc)}\biggr) \\
 &\geq   \frac{1}{t}\biggl( \kappa \ip{\Dc^\ell \Pi Z}{A^0 \Dc^\ell \Pi Z}_{L^2(\Sc)} - (\omega(R) + R)\norm{\Pi Z}_{H^k(\Sc)}
 \norm{\Dc^\ell \Pi Z}_{L^2(\Sc)}\bigr)
\end{align*}
for $\ell \geq 1$ and
\begin{align*}
 \bigg\langle Z \biggl| \frac{1}{t}\Ac \Pi Z) + \frac{1}{t}\Qc\biggl\rangle_{L^2(\Sc)}
 &= \frac{1}{t}\biggl(
 \ip{\Pi Z}{\Ac\Pi Z)\big}_{L^2(\Sc)} +   \ip{\Pi Z}{\Qc}_{L^2(\Sc)}\biggr) \\
 &\geq   \frac{1}{t}\bigl( \kappa \ip{Z}{A^0\Pi Z}_{L^2(\Sc)} - \omega(R)\norm{\Pi Z}_{L^2(\Sc)}^2\biggr).
\end{align*}
where in deriving these we used, in addition to   \eqref{Qcbnd2}-\eqref{Qcbnd4}, the bound \eqref{kappabnd} and the Cauchy-Schwartz 
inequality. 
\end{comment}
%------------------ end comment -------------------------------
\end{rem}

\section{Existence}

\begin{thm} \label{existthm}
Suppose $k\in \Zbb_{\geq 5}$, $\rho_0>0$, the asymptotic flow assumptions \eqref{flowassump.1}-\eqref{flowassump.2} are satisfied
for constants $\Ntt \in \Zbb_{\geq k}$, $R_0>0$ and $\ep \in (0,1/10)$, the constants $\kappa,\nu\in \Rbb_{>0}$
satisfy the inequalities \eqref{kappa-nu-fix1}, and $\zc \in (0,\kappa)$. Then there exist constants $m\in \Zbb_{\geq 1}$ and $\delta>0$ such
that for any $\mathring{V}=(\mathring{V}^K)\in H^{k+1}(\Sc,\Vbb^N)$ satisfying $\norm{\mathring{V}}_{H^{k+1}(\Sc)}<\delta$, there
exists a unique solution
\begin{equation*} 
V=(V^K) \in C^0\bigl((0,1],H^{k+1}(\Sc,\Vbb^N)\bigr)\cap C^1\bigl((0,1],H^{k}(\Sc,\Vbb^N)\bigr)
\end{equation*}
to the GIVP \eqref{MwaveJ.1}-\eqref{MwaveJ.2} for the extended system.   Moreover, the following hold:
\begin{enumerate}[(a)]
\item The  solution $V$ satisfies the bounds
\begin{gather*}
\norm{V_0(t)}_{L^\infty(\Sc)} \lesssim 1, \quad
\norm{V_0(t)}_{H^k(\Sc)}  \lesssim \frac{1}{t^{\ep}}, \quad 
\norm{\Pbb V(t)}_{H^k(\Sc)}  \lesssim t^{\nu},  \\
\norm{\Dc V(t)}_{H^k(\Sc)}  \lesssim \frac{1}{t^{\kappa}}, \quad
\norm{\Pbb V(t)}_{H^{k-1}(\Sc)}  \lesssim t^{\nu+\kappa-\zc} \AND
\norm{\Dc V(t)}_{H^{k-1}(\Sc)}  \lesssim \frac{1}{t^{\zc}}
\end{gather*}
for $t\in (0,1]$.
Additionally, there exists an element $Z^\perp \in H^{k-1}(\Sc,\Wbb)$ satisfying $\Pbb^\perp Z^\perp_0=Z^\perp_0$ such that
\begin{equation*}
\norm{\Pi Z(t)}_{H^{k-1}(\Sc)} + \norm{\Pi^\perp Z(t)-Z^\perp }_{H^{k-1}(\Sc)} \lesssim t^{\kappa-\zc}
\end{equation*}
for $t\in (0,1]$ where $Z$ is determined from $V$ by \eqref{Zdef}. 
\item If, additionally, the initial data $\mathring{V}$ is chosen so that the constraint \eqref{idataconstraint} is satisfied, then the solution $V$
determines a unique classical solution $\ub^K \in C^2(\Mb_{r_0})$,
with $r_0=\rho_0^m$, of the IVP
\begin{align*}
\gb^{\alpha\beta}\nablab_\alpha \nablab_\beta \ub^K &= \ab^{K\alpha\beta}_{IJ}\nablab_\alpha \ub^I \nablab_\beta \ub^J
\quad \text{in $\Mb_{r_0}$,}\\
(\ub^K, \del{\tb}\ub^K) &= (\vb^K,\wb^K_1) \hspace{1.5cm} \text{in $\Sigmab_{r_0}$,}
\end{align*}
where $\ub^K$, $\vb^K$ and $\wb^K$ are determined from $V$ by \eqref{VK2ubK}, \eqref{idatatransB} and \eqref{idatatransC}. Furthermore, the $\ub^K$ satisfy the pointwise bounds
\begin{equation*}
|\ub^K|\lesssim \frac{\rb}{\rb^2-\tb^2}\biggl(1-\frac{\tb}{\rb}\biggr)^{\frac{1}{2}+\nu+\kappa-\zc} \quad \text{in $\Mb_{r_0}$.}
\end{equation*} 
\end{enumerate}
\end{thm}
\begin{proof}
$\;$\\
\noindent\underline{Existence and uniqueness for the extended system:}
Having established that the extended system \eqref{MwaveJ.1} is symmetric hyperbolic, we can, 
since $k>3/2+1$ by assumption, appeal to standard local-in-time
existence and uniqueness results for symmetric hyperbolic systems, e.g. \cite[Ch.~16, Prop.~1.4.]{TaylorIII:1996}, to
conclude the existence of a $t^*\in [0,1)$, which we take to be \textit{maximal}, and a unique solution 
\begin{equation} \label{Vreg}
V=(V^K) \in C^0\bigl((t^*,1],H^{k+1}(\Sc,\Vbb^N)\bigr)\cap C^1\bigl((t^*,1],H^{k}(\Sc,\Vbb^N)\bigr)
\end{equation}
to the IVP \eqref{MwaveJ.1}-\eqref{MwaveJ.2} for given initial data $\mathring{V}=(\mathring{V}^K)\in H^{k+1}(\Sc,\Vbb^N)$, where the maximal time $t^*$ depends on $\mathring{V}$. Next, by \eqref{Ydef}, we have that
\begin{equation*}
Y|_{t=1} =\mathring{V}_0=(\mathring{V}^K_0). 
\end{equation*}
From this, \eqref{WKdef}, \eqref{XKdef}  and \eqref{Zdef},  we see, by choosing the initial data
to satisfy $\norm{\mathring{V}}_{H^{k+1}(\Sc)} < \delta$,
that $\norm{Z(1)}_{H^{k}(\Sc)} < \mathring{C}\delta$
for some positive constant $\mathring{C}>0$ that is independent of $\delta$. We then
fix $R\in (0,R_0]$ and choose $\delta$ small enough to satisfy
\begin{equation} \label{deltabnd}
 \delta < \frac{R}{8 \mathring{C}C_{\text{Sob}}}
\end{equation}
so that
\begin{equation}  \label{Zidatathm}
\norm{Z(1)}_{H^{k}(\Sc)} < \mathring{C}\delta <  \frac{R}{8C_{\text{Sob}}}.
\end{equation}
For $Z$ to be well-defined, it is enough for $Z$ to satisfy
\begin{equation} \label{Zwelldefbnd}
\norm{Z}_{H^k(\Sc)} \leq   \frac{R}{2 C_{\text{Sob}}}.
\end{equation}
This is because this bound will ensure by Sobolev's inequality \eqref{Sobolev} that
\begin{equation*}
\norm{Y}_{L^\infty} \leq C_{\text{Sob}}\norm{Y}_{H^K(\Sc)} \leq C_{\text{Sob}}\norm{Z}_{H^k(\Sc)} \leq \frac{R}{2} <R < R_0,
\end{equation*}
which, by the flow assumptions \eqref{flowassump.1}-\eqref{flowassump.2},
will guarantee that the change of variables \eqref{Ydef} is well-defined and invertible, and hence that $Z$ is well-defined by \eqref{WKdef}, \eqref{XKdef} and \eqref{Zdef}.

To proceed, we let $t_* \in (t^*,0)$ denote the first time such that
\begin{equation} \label{t*def} 
\norm{Z(t_*)}_{H^{k}(\Sc)} =  \frac{R}{2 C_{\text{Sob}}},
\end{equation}
and if there is no such time, then we set $t_*=t^*$. We note that $t_*$ is well-defined by \eqref{deltabnd} and \eqref{Zidatathm},
and  we further note from \eqref{Vreg} and the definition of $Z$ that
\begin{equation*}
Z\in C^0\bigl((t_*,1],H^{k}(\Sc,\Wbb)\bigr)\cap C^1\bigl((t_*,1],H^{k-1}(\Sc,\Wbb)\bigr).
\end{equation*}
Now, since $\Fc(t,1,y,0)=0$ by virtue of $\xi=0$ being a solution of the asymptotic equation \eqref{asympeqn}, it is not difficult to verify that the symmetric hyperbolic equations \eqref{MwaveJ.1} and \eqref{MwaveO} both admit the trivial solution. Because of \eqref{Zidatathm}, we can therefore appeal to the Cauchy stability
property enjoyed by symmetry hyperbolic equations to conclude, by choosing $\delta$ small enough, that $t_*$, where of course $t_*\geq t^*$,
can be made to be as small as we like and that the inequality 
\begin{equation} \label{Zt0bnd}
\max_{t_0\leq t\leq 1}\norm{Z(t)}_{H^{k}(\Sc)} < 2\mathring{C}\delta <  \frac{R}{4 C_{\text{Sob}}}
\end{equation}
is valid for
\begin{equation*}
t_0 = \min\{2 t_*,1/2\}.
\end{equation*}
Recalling that we are free to choose the constant $\sigma>0$, see \eqref{A1bnd}, as small as we like by choosing the constant $m\in \Zbb_{\geq 1}$
sufficiently large, we can, for any given $\sigma_*>0$, arrange, since $\tilde{\ep}>0$ (see \eqref{JcexpB}), that 
\begin{equation} \label{sigmabnd}
\sigma + t^{\tilde{\ep}} < \sigma_*, \quad t\in (0,t_0],
\end{equation}
by choosing $\delta$ small enough to guarantee that $t_0$ is sufficiently small to ensure that this inequality holds. 

In light of Remarks \ref{Jcrem} and \ref{Qcrem}, 
the bounds  \eqref{gammabnd}, \eqref{kappabnd}, \eqref{A1bnd}, and \eqref{sigmabnd},  the relations \eqref{Piproj}-\eqref{PiperpAproj},
the expansion \eqref{JcexpB}, and the estimates \eqref{Qcbnd2}-\eqref{Qcbnd3}, all taken together, show that if the constants $m\in \Zbb_{\geq 1}$ and $\delta>0$ are chosen sufficiently large and small, respectively, and the constants $\kappa,\nu$ are chosen to satisfy \eqref{kappa-nu-fix1}, then the Fuchsian system \eqref{MwaveO}, which $Z$ satisfies, will, after the simple time transformation $t\mapsto -t$, satisfy all the required
assumptions needed to apply the time rescaled version, see \cite[\S 3.4.]{BOOS:2020} and the remark below, of Theorem 3.8.~from \cite{BOOS:2020}.

\begin{rem}  \label{constantsrem}
From the discussion from Section 3.4.~of
\cite{BOOS:2020} and Section \ref{coefprops} of this article, it not difficult to see that the appropriate rescaling
power $p$, see equation (3.106) in \cite{BOOS:2020}, in the current context is
\begin{equation} \label{rescaledef}
p= \kappa-2\ep, 
\end{equation}
which, we note, by \eqref{kappa-nu-fix1}, satisfies the required bounds $0<\kappa-2\ep\leq 1$. We further note from Theorem 3.8.~ from \cite{BOOS:2020}, see also \cite[\S 3.4.]{BOOS:2020},
that parameter
$\zeta$ defined by equation (3.59) of \cite{BOOS:2020}, which is involved in determining the decay of solutions, is,
in the current context, determined by
\begin{equation} \label{zetadef}
\zeta = \kappa - \zc
\end{equation} 
where $\zc>0$ can be made as small as we like by choosing the constant $m$ large enough and the constants $R,t_0$ small
enough to ensure that  $\sigma_*$ and $\norm{Z}_{H^K(\Sc)}$  are sufficiently small.
\end{rem}
We therefore
conclude from the proof of Theorem 3.8.~from \cite{BOOS:2020} that $Z$, which solves \eqref{MwaveO}, satisfies an energy estimate of the form
\begin{equation} \label{Zenergy}
\norm{Z(t)}_{H^k(\Sc)}^2 + \int^{t_0}_{t} \frac{1}{\tau} \norm{\Pi Z(\tau)}^2_{H^k(\Sc)}\, d\tau
\leq C_E^2\norm{Z(t_0)}^2
\end{equation}
for all $t\in (t_*,t_0]$. By Gr\" {o}nwall's inequality and \eqref{Zidatathm}, we then have
\begin{equation} \label{Zbnd1}
\sup_{t\in (t_*,t_0)}\norm{Z(t)}_{H^k(\Sc)} \leq e^{C_E(t_*-t_0)} \norm{Z(t_0)}_{H^k(\Sc)} < e^{C_E(t_*-t_0)}\mathring{C}\delta.
\end{equation} 
Choosing $\delta$ now, by shrinking it if necessary, to satisfy
$\delta < \frac{R}{3\mathring{C}C_{\text{Sob}}e^{C_E(t_*-t_0)}}$
in addition to \eqref{deltabnd}, the bounds \eqref{Zt0bnd} and \eqref{Zbnd1} implies that
\begin{equation} \label{Zbnd2}
\sup_{t\in (t_*,1)}\norm{Z(t)}_{H^k(\Sc)} <  \frac{R}{3C_{\text{Sob}}}.
\end{equation}
From this inequality and the definition \eqref{t*def} for $t_*$, we conclude that $t_*=t^*$.

Now, from \eqref{flowassump1}, \eqref{flowassump2}, Sobolev's inequality, and the Moser estimates (e.g. \cite[Ch.~13, Prop.~3.9.]{TaylorIII:1996}),
we see from \eqref{Ydef} and \eqref{Zdef} that $V_0$ can be bounded by
\begin{equation} \label{V0bndA}
\norm{V_0(t)}_{H^k(\Sc)} \leq \frac{1}{t^\ep}C(\norm{Z(t)}_{H^k(\Sc)})\norm{Z(t)}_{H^k(\Sc)}
\end{equation}
for $Z$ satisfying \eqref{Zwelldefbnd},
while we see from \eqref{XKdef}, \eqref{Zdef} and \eqref{Pidef} that $\Pbb V(t)$ is bounded by
\begin{equation}\label{PbbVbndA}
\norm{\Pbb V(t)}_{H^s(\Sc)} \leq t^\nu \norm{\Pi Z(t)}_{H^s(\Sc)}, \quad s\in \Zbb_{\geq 0}.
\end{equation}
Since $t_*=t^*$, the estimates \eqref{Zbnd2}, \eqref{V0bndA} and \eqref{PbbVbndA} imply that  
$\norm{V(t)}_{H^k(\Sc)}$ is
finite for any $t\in (t^*,0)$. By the maximality of $t^*$ and the continuation principle for symmetric hyperbolic equations, we conclude that $t^*=0$,
which establishes the existence of solutions to the extended IVP \eqref{MwaveJ.1}-\eqref{MwaveJ.2} on the spacetime
region $(0,1]\times \Sc$.

\bigskip

\noindent\underline{Uniform bounds for $V$:}
From \eqref{WKdef}, \eqref{Zdef}, \eqref{V0bnd}, \eqref{Zbnd2}, \eqref{V0bndA} and  \eqref{PbbVbndA}, we see that the estimates
\begin{gather*}
\norm{V_0(t)}_{L^\infty(\Sc)} \lesssim \omega(\delta), \quad
\norm{V_0(t)}_{H^k(\Sc)}  \lesssim \frac{1}{t^{\ep}}\delta, \quad 
\norm{\Pbb V(t)}_{H^k(\Sc)}  \lesssim t^{\nu}\delta \AND
\norm{\Dc V(t)}_{H^k(\Sc)}  \lesssim \frac{1}{t^{\kappa}}\delta
\end{gather*}
hold for $t\in (0,1]$. Furthermore, in view of the Remark \ref{constantsrem}, see in particular, \eqref{zetadef},  the coefficient properties from Section \ref{coefprops}, and the fact that $\kappa \in (0,1/3]$, we conclude from Theorem 3.8.~and Section 3.4.~of \cite{BOOS:2020} that, for any
fixed $\zc>0$, there exists, provided that  $m$ and $\delta$ are chosen sufficiently large and small respectively, an element $Z^\perp \in H^{k-1}(\Sc,\Wbb)$ satisfying $\Pbb^\perp Z^\perp_0=Z^\perp_0$
such that
\begin{equation*}%\label{ZdecayA}
\norm{\Pi Z(t)}_{H^{k-1}(\Sc)} + \norm{\Pi^\perp Z(t)-Z^\perp }_{H^{k-1}(\Sc)} \lesssim t^{\kappa-\zc}
\end{equation*}
for $t\in (0,1]$. With the help of the above inequality, \eqref{WKdef}, \eqref{Zdef}, \eqref{Pidef} and \eqref{PbbVbndA},
we conclude that $V$ also satisfies
\begin{equation} \label{VdecayA}
\norm{\Pbb V(t)}_{H^{k-1}(\Sc)}  \lesssim t^{\nu+\kappa-\zc} \AND
\norm{\Dc V(t)}_{H^{k-1}(\Sc)}  \lesssim \frac{1}{t^{\zc}}
\end{equation}
for $t\in (0,1]$.

\bigskip

\noindent \underline{Existence for the wave equations \eqref{Mbwave}:} Letting $r_0=\rho_0^m$, we know from the discussion contained in Section \ref{extendedsec},
that if the initial data $\mathring{V}$ is chosen to satisfy the constraints \eqref{idataconstraint} on the spacelike hypersurface $\Sigma_{r_0}$, then the solution  $V=(V^K_0,V^K_1,V^K_\Lambda,V^K_4)$ to the extended system \eqref{MwaveJ.1} determines a classical
solution $\ub^K$ of the semilinear wave equations \eqref{Mbwave} on $\Mb_{r_0}$ via the formula
\eqref{VK2ubK}. Moreover, this solution is uniquely determined by the initial data on $\Sigma_{r_0}$ that is obtained from the restriction of the initial data $\mathring{V}$ to the initial hypersurface $\Sigma_{r_0}$ and the transformation formulas
\eqref{idatatransA} and \eqref{idatatransB}.
To complete the proof, we note from
\eqref{VK2uK}, Sobolev's inequality, the decay estimate \eqref{VdecayA}, and \eqref{tbrb2tr} that each $\ub^K$ satisfies the pointwise bound
\begin{equation*}
|\ub^K|\lesssim \frac{\rb}{\rb^2-\tb^2}\biggl(1-\frac{\tb}{\rb}\biggr)^{\frac{1}{2}+\nu+\kappa-\zc} \quad \text{in $\Mb_{r_0}$.}
\end{equation*}
\end{proof}

\begin{cor} \label{existcor}
Suppose $k\in \Zbb_{\geq 5}$, $\rho_0>0$, $\zc>0$ and the bounded weak null condition (see Definition \ref{bwnc}) holds. Then there exist constants $m\in \Zbb_{\geq 1}$ and $\delta>0$ such
that for any $\mathring{V}=(\mathring{V}^K)\in H^{k+1}(\Sc,\Vbb^N)$ satisfying $\norm{\mathring{V}}_{H^k(\Sc)}<\delta$, there
exists a unique solution
\begin{equation*} 
V=(V^K) \in C^0\bigl((0,1],H^{k+1}(\Sc,\Vbb^N)\bigr)\cap C^1\bigl((0,1],H^{k}(\Sc,\Vbb^N)\bigr)
\end{equation*}
to the IVP \eqref{MwaveJ.1}-\eqref{MwaveJ.2}.   Moreover, the following hold:
\begin{enumerate}[(a)]
\item  The  solution $V$ satisfies the uniform bounds
\begin{gather*}
\norm{V_0(t)}_{L^\infty(\Sc)} \lesssim 1, \quad
\norm{V_0(t)}_{H^k(\Sc)}+\norm{\Dc V(t)}_{H^k(\Sc)}  \lesssim \frac{1}{t^{\zc}} \AND 
\norm{\Pbb V(t)}_{H^k(\Sc)}  \lesssim t^{\frac{1}{2}-\zc} 
\end{gather*}
for $t\in (0,1]$.
\item If, additionally, the initial data $\mathring{V}$ is chosen so that the constraint \eqref{idataconstraint} is satisfied, then the solution $V$
determines a unique classical solution $\ub^K \in C^2(\Mb_{r_0})$, with $r_0=\rho_0^m$, of the IVP
\begin{align*}
\gb^{\alpha\beta}\nablab_\alpha \nablab_\beta \ub^K &= \ab^{K\alpha\beta}_{IJ}\nablab_\alpha \ub^I \nablab_\beta \ub^J
\quad \text{in $\Mb_{r_0}$,}\\
(\ub^K, \del{\tb}\ub^K) &= (\vb^K,\wb^K) \hspace{1.5cm} \text{in $\Sigmab_{r_0}$,}
\end{align*}
where $\ub^K$, $\vb^K$ and $\wb^K$ are determined from $V$ by \eqref{VK2ubK}, \eqref{idatatransB} and \eqref{idatatransC}. Furthermore, the $\ub^K$ satisfy the pointwise bounds
\begin{equation*}
|\ub^K|\lesssim \frac{\rb}{\rb^2-\tb^2}\biggl(1-\frac{\tb}{\rb}\biggr)^{1-\zc} \quad \text{in $\Mb_{r_0}$.}
\end{equation*} 
\end{enumerate}
\end{cor}

\begin{proof}
By Proposition \ref{asympprop}, we know that the asymptotic flow
satisfies the flow assumptions \eqref{flowassump.1}-\eqref{flowassump.2} for some $R_0>0$ and any $\ep \in (0,1/10]$. Fixing $\ep \in (0,1/11)$,
we set $\zc= \ep$, $\nu = \frac{1}{2}-5\zc$ and  $\kappa = 3\zc$. 
It is then not difficult to verify that these choices for $\zc$, $\nu$ and $\kappa$ satisfy the inequalities \eqref{kappa-nu-fix1}
and  $0<\zc < \kappa$. The proof now follows directly from Theorem \ref{existthm}.
\end{proof}

\bigskip

\noindent \textit{Acknowledgements:}
This work was partially supported by the Australian Research Council grant DP170100630. 
J. A. Olvera-Santamar\'{i}a  also acknowledges support from the CONACYT grant 709315.   
\appendix

\section{Indexing conventions\label{indexing}}
Below is a summary of the indexing conventions that are employed throughout this article:
\medskip
\begin{center}
\begin{tabular}{|l|c|c|l|} \hline
Alphabet & Examples & Index range & Index quantities  \\ \hline
Lowercase Greek & $\mu,\nu,\gamma$ & $0,1,2,3$ & spacetime coordinate components, e.g. $(x^\mu)=(t,r,\theta,\phi)$ \\ \hline
Uppercase Greek &$\Lambda,\Sigma,\Omega$ & $2,3$, & spherical coordinate components, e.g. $(x^\Lambda)=(\theta,\phi)$ \\ \hline
Lowercase Latin &$i,j,k $ & $1,2,3$ & spatial coordinate components, e.g. $(y^i)=(\rho,\theta,\phi)$ \\ \hline
Uppercase Latin &$I,J,K$ & $1$ to $N$ & wave equation indexing, e.g. $u^I$  \\ \hline
Lowercase Calligraphic &$\qc,\pc,\rc$ & 0,1 & time and radial coordinate components, e.g. $(x^\qc)=(t,r)$\\ \hline
Uppercase Calligraphic &$\Ic,\Jc,\Kc$ & 0,1,2,3,4 & first order wave formulation indexing, e.g. $V^K_\Ic$  \\ \hline
\end{tabular}
\end{center}

\section{Conformal Transformations\label{ctrans}}
In this section, we recall a number of formulas that govern the transformation laws for geometric objects under a conformal transformation that will be needed for our application to wave equations. Under a  
conformal transformation of the form
\begin{equation} \label{gtrans}
\gt_{\mu\nu} = \Omega^2 g_{\mu\nu},
\end{equation}
the Levi-Civita connection $\nablat_\mu$ and $\nabla_\mu$ of $\gt_{\mu\nu}$ and $g_{\mu\nu}$, respectively, are related by
\begin{equation*} %\label{christtrans}
\nablat_{\mu}\omega_\nu = \nabla_\mu\omega_\nu - \Cc_{\mu\nu}^\lambda \omega_\lambda,
\end{equation*}
where 
\begin{equation*}
\Cc_{\mu\nu}^\lambda = 2\delta^\lambda_{(\mu}\nabla_{\nu)}\ln(\Omega)
-g_{\mu\nu}g^{\lambda\sigma}\nabla_\sigma \ln(\Omega).
\end{equation*}
Using this, it can be shown that the wave operator transforms as
\begin{equation}\label{wavetransA}
\gt^{\mu\nu}\nablat_\mu\nablat_\nu \ut - \frac{n-2}{4(n-1)}\Rt \ut = \Omega^{-1-\frac{n}{2}}
\biggl(g^{\mu\nu}\nabla_\mu\nabla_\nu u - \frac{n-2}{4(n-1)}R u\biggr)
\end{equation}
where $\Rt$ and $R$ are the Ricci curvature scalars of $\gt$ and $g$, respectively, $n$ is the dimension of
spacetime, and
\begin{equation} \label{utransA}
\ut = \Omega^{1-\frac{n}{2}}u.
\end{equation} 
Assuming now that the scalar functions $\ut^K$ satisfy the system of wave equations
\begin{equation} \label{wavetransB}
\gt^{\mu\nu}\nablat_\mu\nablat_\nu \ut^K - \frac{n-2}{4(n-1)}\Rt \ut^K =\ft^K,
\end{equation}
it then follows immediately from \eqref{wavetransA} and
\eqref{utransA} that the scalar functions
\begin{equation} \label{utrans}
u^K = \Omega^{\frac{n}{2}-1}\ut^K
\end{equation}
satisfy the conformal system of wave equations given by
\begin{equation} \label{wavetransC}
g^{\mu\nu}\nabla_\mu\nabla_\nu u^K - \frac{n-2}{4(n-1)} R u^K =f^K
\end{equation}
where
\begin{equation} \label{ftransA}
f^K = \Omega^{1+\frac{n}{2}}\ft^K.
\end{equation}
Specializing to source terms $\ft^K$ that are quadratic in the derivatives, that is, of the form
\begin{equation} \label{ftransB}
\ft^K = \at^{K\mu\nu}_{IJ}\nablat_\mu \ut^I \nablat_\nu \ut^J,
\end{equation}
a short calculation using \eqref{gtrans} and \eqref{utrans} shows that the corresponding conformal source $f^K$,
defined by \eqref{ftransA}, is given by
\begin{align} 
f^K = \at^{K\mu\nu}_{IJ}\biggl(\Omega^{3-\frac{n}{2}}\nabla_\mu u^I\nabla_\nu u^J& +\bigg(\frac{n}{2}-1\biggr)\Omega^{4-\frac{n}{2}}\bigl(\nabla_\mu\Omega^{-1} u^I \nabla_\nu u^J+ \nabla_\mu u^I \nabla_\nu\Omega^{-1} u^J \bigr) \notag\\
&\qquad  + \bigg(1-\frac{n}{2}\biggr)^2  \Omega^{5-\frac{n}{2}}\nabla_\mu\Omega^{-1} \nabla_\nu\Omega^{-1} u^I u^J\biggr).\label{ftransC}
\end{align}

\bibliographystyle{amsplain}

\providecommand{\bysame}{\leavevmode\hbox to3em{\hrulefill}\thinspace}
\providecommand{\MR}{\relax\ifhmode\unskip\space\fi MR }
\providecommand{\MRhref}[2]{%
  \href{http://www.ams.org/mathscinet-getitem?mr=#1}{#2}
}
\providecommand{\href}[2]{#2}

\end{document}